\documentclass[leqno, 12pt]{article}

\usepackage{amsthm}
\usepackage{amssymb}
\usepackage{amsmath}
\usepackage[english]{babel}
\usepackage[utf8]{inputenc}
\usepackage[T1]{fontenc}
\usepackage[normalem]{ulem}
\usepackage{xhfill}
\usepackage{graphicx}
\usepackage{fullpage}
\usepackage{enumitem}
\usepackage{calc}
\usepackage{caption}
\usepackage{subcaption}
\usepackage{hyperref}
\usepackage{times}
\usepackage{float}
\usepackage{marvosym}
\usepackage{hyperref}
\usepackage{tikz}
\usetikzlibrary{shapes,fit,positioning,calc,matrix,arrows}
\tikzset{
  optree/.style={scale=.5,thick,grow'=up,level distance=10mm,inner sep=1pt},
  comp/.style={draw=none,circle,fill,line width=0,inner sep=0pt},
  dot/.style={draw,circle,fill,inner sep=0pt,minimum width=3pt},
  circ/.style={draw,circle,inner sep=1pt,minimum width=4mm},
  emptycirc/.style={draw,circle,inner sep=1pt,minimum width=2mm},  
  root/.style={level distance=10mm,inner sep=1pt},
  leaf/.style={draw=none,circle,fill,line width=0,inner sep=0pt},
  nodot/.style={draw,circle,inner sep=1pt},
}
\usepackage{tikz-cd}

\newtheorem*{Conjecture*}{Conjecture}
\newtheorem*{Theorem*}{Theorem}
\newtheorem{Theorem}{Theorem}[section]

\newtheorem{Proposition}{Proposition}[section]
\newtheorem{Lemma}{Lemma}[section]

\theoremstyle{definition}
\newtheorem{Definition}[Proposition]{Definition}

\theoremstyle{remark}
\newtheorem{Remark}{Remark}[section]
\newtheorem{Example}[Remark]{Example}

\DeclareFontFamily{U}{wncy}{}
\DeclareFontShape{U}{wncy}{m}{n}{<->wncyr10}{}
\DeclareSymbolFont{mcy}{U}{wncy}{m}{n}
\DeclareMathSymbol{\Sh}{\mathord}{mcy}{"58}

\newcommand{\id}{\textnormal{id}}

\makeatletter
\newcommand{\oset}[2]{{\mathpalette\o@set{{#1}{#2}}}}
\newcommand{\o@set}[2]{\o@@set{#1}#2}
\newcommand{\o@@set}[3]{%
  \vbox{\offinterlineskip
    \ialign{\hfil##\hfil\cr
      $\m@th\o@set@demote{#1}#2$\cr
      \noalign{\vskip0.2pt}
      $\m@th#1#3$\cr
    }%
  }%
}
\newcommand{\o@set@demote}[1]{%
  \ifx#1\displaystyle\scriptstyle\else
  \ifx#1\textstyle\scriptstyle\else
  \scriptscriptstyle\fi\fi
}
\makeatother

\newlength{\depthofsumsign}
\newlength\myheight

\tikzset{cross/.style={cross out, draw=black, minimum size=2*(#1-\pgflinewidth), inner sep=0pt, outer sep=0pt},
cross/.default={1pt}}

\begin{document}

\title{Braces on the cohomology of noncrossing 2-partitions}
\author{\textsc{Paul Laubie}}
\date{}
\maketitle

\begin{abstract}
	We show that the operadic structure on the cohomology of the poset of noncrossing 2-partitions is isomorphic to the Brace operad.
\end{abstract}

\section*{Introduction}

Partitions are one of the most useful and most studied combinatorial objects. 
They admit a natural partial order, and thus, one may compute the homology (or cohomology) of this poset.
It was shown by Joyal that this cohomology has a natural Lie algebra structure; it is, in fact, the Lie operad, a universal object encoding Lie algebras.
A natural follow-up question is whether this result generalizes to other similar objects.
In a recent article, Delcroix-Oger and Dupont define the notion of operadic partition posets, generalising the poset of partitions, and show that each of those admits a natural structure of operad on their cohomology.\\
In this article, we compute the actual structure on the poset of noncrossing 2-partitions.
The noncrossing 2-partitions are particularly interesting combinatorial objects since they are known to be equivalent to parking functions.
In this article, we show that the algebraic structure on their cohomology is a structure of Brace algebra (actually the Brace operad).
The notion of Brace algebra was first introduced by Gerstenhaber and Voronov as a structure induced by the partial composition of multilinear functions; moreover, any Brace algebra is in particular a Lie algebra.
It was later shown by Chapoton that free Brace algebras admit a combinatorial description with rooted planar trees.\\
An interesting consequence of getting the Brace operad is that the cohomology of the augmented poset of noncrossing 2-partitions has a natural Brace algebra structure.\\ 

The current article is organised the following way:
\begin{itemize}
	\item First, we give some recollection on the theory of species of Joyal and on the theory of operads.
		We study the specific species of planar rooted trees and the Brace operad.
	\item We then introduce noncrossing 2-partitions, their cohomology, and an explicit isomorphism between their cohomology and the rooted planar trees.
	\item Finally, we show that, once endowed with its operadic structure, the cohomology of the poset of noncrossing 2-partitions is isomorphic to the Brace operad.
\end{itemize}

\tableofcontents

\section{Preliminaries}

\subsection{Species}
We freely use the language of Species throughout the article.
We refer the reader to \cite{MonFSpe,TLSpe,CombSpe,AlgOp} for an introduction to the theory of species and of algebraic operads.
Let us recall the basic definitions of the theory of species and of operads:

Let $\mathbb{B}$ the groupoid of finite sets with bijections and $\mathrm{Set}$ the category of sets.
A \emph{combinatorial species} $\mathrm{S}$ is a functor $\mathbb{B}\to\mathrm{Set}$; equivalently, this is a sequence of sets $(\mathrm{S}(n))_{n\in\mathbb{N}}$ endowed with an action of $\mathfrak{S}_n$ on $\mathrm{S}(n)$.
Similarly, let $\mathrm{Vect}$ the category of vector spaces, a \emph{linear species} $\mathcal{S}$ is a functor $\mathbb{B}\to\mathrm{Vect}$ which is equivalently a sequence of representations of $\mathfrak{S}_n$.
Let $\mathrm{gr}\mathrm{Vect}$ the category of graded vector spaces, a \emph{graded linear species} is a functor $\mathbb{B}\to\mathrm{gr}\mathrm{Vect}$.
Combinatorial and linear species behave like a ``categorification'' of exponential generating series. 
The exponential generating series associated with a combinatorial species $\mathrm{S}$ is:
$$f_\mathrm{S}=\sum \frac{|\mathrm{S}(n)|}{n!}t^n$$
with notation $\mathrm{S}(n)=\mathrm{S}(\{1,\dots,n\})$. 
Thus, the exponential generation series is well-defined only if $\mathrm{S}(n)$ is finite for each $n\in\mathbb{N}$, or equivalently only if $\mathrm{S}(N)$ is finite for each finite set $N$.
The exponential generating series associated with a linear species $\mathcal{S}$ is:
$$f_\mathcal{S}=\sum \frac{\dim(\mathcal{S}(n))}{n!}t^n$$
Thus, the exponential generation series is well-defined only if $\mathcal{S}(n)$ is finite dimensional for each $n\in\mathbb{N}$, or equivalently only if $\mathcal{S}(N)$ is finite dimensional for each finite set $N$.

Species admit a sum, a product, and a composition (the plethysm) defined the same as the sum, product, and composition of formal power series.
Let $\mathrm{S},\mathrm{R}$ two combinatorial species:
$$\left(\mathrm{S}+\mathrm{R} \right)(N)=\mathrm{S}(N)\uplus\mathrm{R}(N), \qquad \left(\mathrm{S}\times\mathrm{R} \right)(N)=\biguplus_{I\sqcup J=N}\mathrm{S}(I)\times\mathrm{R}(J)$$
$$\left(\mathrm{S}\circ\mathrm{R} \right)(N)=\biguplus_{P\vdash N}\mathrm{S}(P)\times\prod_{q\in P}\mathrm{R}(q) \quad \text{ for } \mathrm{R}(\emptyset)=\emptyset$$
Here, we used the notation $\uplus$ for the coproduct in the category of sets, which is the ``forced disjoint'' union, meaning that even if $A$ and $B$ have common elements, the subsets $A\hookrightarrow A\uplus B$ and $B\hookrightarrow A\uplus B$ are disjoint.
We used the notation $\sqcup$ for the usual union, when it is actually disjoint.
Finally, we used the notation $P\vdash N$ for $P$ a partition of $N$.

Those operations correspond to the sum, the product, and the composition of formal power series, and they satisfy the same compatibility relations; they are associative, the sum distributes over the product, etc ...
Let us point out that the last line is a categorical version of the Faà di Bruno formula of the coefficients of the composition of two formal power series.

We have the same formulæ for linear species by replacing $\uplus$ by $\oplus$ and $\times$ by $\otimes$.
Let $\mathcal{S},\mathcal{R}$ two combinatorial species:
$$\left(\mathcal{S}\oplus\mathcal{R} \right)(N)=\mathcal{S}(N)\oplus\mathcal{R}(N), \qquad \left(\mathcal{S}\otimes\mathcal{R} \right)(N)=\bigoplus_{I\sqcup J=N}\mathcal{S}(I)\otimes\mathcal{R}(J)$$
$$\left(\mathcal{S}\circ\mathcal{R} \right)(N)=\bigoplus_{P\vdash N}\mathcal{S}(P)\otimes\bigotimes_{q\in P}\mathcal{R}(q) \quad \text{ for } \mathcal{R}(\emptyset)=0$$

A \emph{set operad} (resp. \emph{algebraic operad}) is a monoidal object with respect to the plethysm in the category of combinatorial species (resp. linear species). 
Thus, an operad $(\mathcal{P},\gamma,\mathrm{e})$ is the data of a species $\mathcal{P}$, a \emph{composition map} $\gamma:\mathcal{P}\circ\mathcal{P}\to \mathcal{P}$, and a unit $\mathrm{e}\in \mathcal{P}(1)$ satisfying the usual associativity, and unitality relations.
We recall that the data of a composition map $\gamma$ is equivalent to the data of \emph{partial compositions} $\circ_i$.
Let us use the phrase ``$j$ is an input of $f$'' to mean that $f\in\mathcal{P}(J)$ with $j\in J$.
The \emph{partial compositions} are equivariant maps $\circ_i:(f,g)\mapsto f\circ_i g$ for $i$ an input of $f$, such that:
$$\begin{array}{c|cc}
	(f\circ_i g)\circ_j h = & (f\circ_j h)\circ_i g & \text{ if $j$ is an input of $f$}\\
	& f\circ_i(g\circ_j h) & \text{ if $j$ is an input of $g$}
\end{array}$$
And $\mathrm{e}\circ_1 f=f\circ_i \mathrm{e}=f$.

\subsection{Rooted planar trees}

Let us introduce one of the two main combinatorial objects of this article.

\begin{Definition}
	A (finite) \emph{rooted planar tree} is inductively defined as being the data of a vertex called the \emph{root} and a totally ordered set of rooted planar trees called the \emph{children} of this vertex.
	The set of vertices $v(\tau)$ of a planar rooted tree $\tau=(r;\tau_1,\dots,\tau_n)$ is $v(\tau)=\{r\}\sqcup\bigsqcup v(\tau_i)$.
	A vertex is a \emph{leaf} if its set of children is empty.
	The \emph{height} of the rooted planar tree $\tau$ is $h(\tau)=1+\max(h(\tau_i))$, with the height of the single vertex tree being $0$.
	This is exactly the length of the longest path from the root to a leaf.\\
	The species of rooted planar trees denoted $\mathrm{PT}$ is the species such that $\mathrm{PT}(N)$ is the set of planar rooted trees $\tau$ with $v(\tau)=N$, and the bijections act by relabelling the vertices.\\
	Let $\mathrm{C}$ the combinatorial species of \emph{rooted planar corollas}, which are the rooted planar trees of height exactly $1$.
\end{Definition}

\begin{Example}
	Let us explicitely describe $\mathrm{PT}(3)$:
	$$\left\{\,
		\vcenter
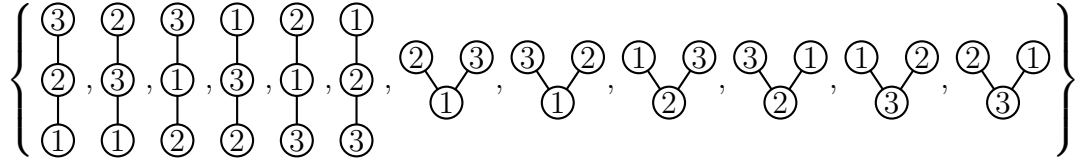
{\hbox{\begin{tikzpicture}[scale=0.4]
			\draw[thick] (0, 0)--(0, 2);
			\draw[thick] (0, 2)--(0, 4);
			\draw[fill=white, thick] (0, 0) circle [radius=15pt];
			\draw[fill=white, thick] (0, 2) circle [radius=15pt];
			\draw[fill=white, thick] (0, 4) circle [radius=15pt];
			\node at (0, 0) {$1$};
			\node at (0, 2) {$2$};
			\node at (0, 4) {$3$};
		\end{tikzpicture}}}\;,
		\vcenter{\hbox{\begin{tikzpicture}[scale=0.4]
			\draw[thick] (0, 0)--(0, 2);
			\draw[thick] (0, 2)--(0, 4);
			\draw[fill=white, thick] (0, 0) circle [radius=15pt];
			\draw[fill=white, thick] (0, 2) circle [radius=15pt];
			\draw[fill=white, thick] (0, 4) circle [radius=15pt];
			\node at (0, 0) {$1$};
			\node at (0, 2) {$3$};
			\node at (0, 4) {$2$};
		\end{tikzpicture}}}\;,
		\vcenter{\hbox{\begin{tikzpicture}[scale=0.4]
			\draw[thick] (0, 0)--(0, 2);
			\draw[thick] (0, 2)--(0, 4);
			\draw[fill=white, thick] (0, 0) circle [radius=15pt];
			\draw[fill=white, thick] (0, 2) circle [radius=15pt];
			\draw[fill=white, thick] (0, 4) circle [radius=15pt];
			\node at (0, 0) {$2$};
			\node at (0, 2) {$1$};
			\node at (0, 4) {$3$};
		\end{tikzpicture}}}\;,
		\vcenter{\hbox{\begin{tikzpicture}[scale=0.4]
			\draw[thick] (0, 0)--(0, 2);
			\draw[thick] (0, 2)--(0, 4);
			\draw[fill=white, thick] (0, 0) circle [radius=15pt];
			\draw[fill=white, thick] (0, 2) circle [radius=15pt];
			\draw[fill=white, thick] (0, 4) circle [radius=15pt];
			\node at (0, 0) {$2$};
			\node at (0, 2) {$3$};
			\node at (0, 4) {$1$};
		\end{tikzpicture}}}\;,
		\vcenter{\hbox{\begin{tikzpicture}[scale=0.4]
			\draw[thick] (0, 0)--(0, 2);
			\draw[thick] (0, 2)--(0, 4);
			\draw[fill=white, thick] (0, 0) circle [radius=15pt];
			\draw[fill=white, thick] (0, 2) circle [radius=15pt];
			\draw[fill=white, thick] (0, 4) circle [radius=15pt];
			\node at (0, 0) {$3$};
			\node at (0, 2) {$1$};
			\node at (0, 4) {$2$};
		\end{tikzpicture}}}\;,
		\vcenter{\hbox{\begin{tikzpicture}[scale=0.4]
			\draw[thick] (0, 0)--(0, 2);
			\draw[thick] (0, 2)--(0, 4);
			\draw[fill=white, thick] (0, 0) circle [radius=15pt];
			\draw[fill=white, thick] (0, 2) circle [radius=15pt];
			\draw[fill=white, thick] (0, 4) circle [radius=15pt];
			\node at (0, 0) {$3$};
			\node at (0, 2) {$2$};
			\node at (0, 4) {$1$};
		\end{tikzpicture}}}\;,
		\vcenter{\hbox{\begin{tikzpicture}[scale=0.4]
			\draw[thick] (0, 0.5)--(-1, 2);
			\draw[thick] (0, 0.5)--(1, 2);
			\draw[fill=white, thick] (0, 0.5) circle [radius=15pt];
			\draw[fill=white, thick] (-1, 2) circle [radius=15pt];
			\draw[fill=white, thick] (1, 2) circle [radius=15pt];
			\node at (0, 0.5) {$1$};
			\node at (-1, 2) {$2$};
			\node at (1, 2) {$3$};
		\end{tikzpicture}}},
		\vcenter{\hbox{\begin{tikzpicture}[scale=0.4]
			\draw[thick] (0, 0.5)--(-1, 2);
			\draw[thick] (0, 0.5)--(1, 2);
			\draw[fill=white, thick] (0, 0.5) circle [radius=15pt];
			\draw[fill=white, thick] (-1, 2) circle [radius=15pt];
			\draw[fill=white, thick] (1, 2) circle [radius=15pt];
			\node at (0, 0.5) {$1$};
			\node at (-1, 2) {$3$};
			\node at (1, 2) {$2$};
		\end{tikzpicture}}},
		\vcenter{\hbox{\begin{tikzpicture}[scale=0.4]
			\draw[thick] (0, 0.5)--(-1, 2);
			\draw[thick] (0, 0.5)--(1, 2);
			\draw[fill=white, thick] (0, 0.5) circle [radius=15pt];
			\draw[fill=white, thick] (-1, 2) circle [radius=15pt];
			\draw[fill=white, thick] (1, 2) circle [radius=15pt];
			\node at (0, 0.5) {$2$};
			\node at (-1, 2) {$1$};
			\node at (1, 2) {$3$};
		\end{tikzpicture}}},
		\vcenter{\hbox{\begin{tikzpicture}[scale=0.4]
			\draw[thick] (0, 0.5)--(-1, 2);
			\draw[thick] (0, 0.5)--(1, 2);
			\draw[fill=white, thick] (0, 0.5) circle [radius=15pt];
			\draw[fill=white, thick] (-1, 2) circle [radius=15pt];
			\draw[fill=white, thick] (1, 2) circle [radius=15pt];
			\node at (0, 0.5) {$2$};
			\node at (-1, 2) {$3$};
			\node at (1, 2) {$1$};
		\end{tikzpicture}}},
		\vcenter{\hbox{\begin{tikzpicture}[scale=0.4]
			\draw[thick] (0, 0.5)--(-1, 2);
			\draw[thick] (0, 0.5)--(1, 2);
			\draw[fill=white, thick] (0, 0.5) circle [radius=15pt];
			\draw[fill=white, thick] (-1, 2) circle [radius=15pt];
			\draw[fill=white, thick] (1, 2) circle [radius=15pt];
			\node at (0, 0.5) {$3$};
			\node at (-1, 2) {$1$};
			\node at (1, 2) {$2$};
		\end{tikzpicture}}},
		\vcenter{\hbox{\begin{tikzpicture}[scale=0.4]
			\draw[thick] (0, 0.5)--(-1, 2);
			\draw[thick] (0, 0.5)--(1, 2);
			\draw[fill=white, thick] (0, 0.5) circle [radius=15pt];
			\draw[fill=white, thick] (-1, 2) circle [radius=15pt];
			\draw[fill=white, thick] (1, 2) circle [radius=15pt];
			\node at (0, 0.5) {$3$};
			\node at (-1, 2) {$2$};
			\node at (1, 2) {$1$};
		\end{tikzpicture}}}
	\right\}$$
	For corollas, we have that: 
	\begin{multline*}
		\mathrm{C}(0)=\emptyset, \mathrm{C}(1)=\emptyset, \mathrm{C}(2)=\left\{
			\vcenter{\hbox{\begin{tikzpicture}[scale=0.4]
				\draw[thick] (0, 0)--(0, 2);
				\draw[fill=white, thick] (0, 0) circle [radius=15pt];
				\draw[fill=white, thick] (0, 2) circle [radius=15pt];
				\node at (0, 0) {$1$};
				\node at (0, 2) {$2$};
			\end{tikzpicture}}},
			\vcenter{\hbox{\begin{tikzpicture}[scale=0.4]
				\draw[thick] (0, 0)--(0, 2);
				\draw[fill=white, thick] (0, 0) circle [radius=15pt];
				\draw[fill=white, thick] (0, 2) circle [radius=15pt];
				\node at (0, 0) {$2$};
				\node at (0, 2) {$1$};
			\end{tikzpicture}}}
		\right\},\\
		\mathrm{C}(3)=\left\{ 
			\vcenter{\hbox{\begin{tikzpicture}[scale=0.4]
				\draw[thick] (0, 0.5)--(-1, 2);
				\draw[thick] (0, 0.5)--(1, 2);
				\draw[fill=white, thick] (0, 0.5) circle [radius=15pt];
				\draw[fill=white, thick] (-1, 2) circle [radius=15pt];
				\draw[fill=white, thick] (1, 2) circle [radius=15pt];
				\node at (0, 0.5) {$1$};
				\node at (-1, 2) {$2$};
				\node at (1, 2) {$3$};
			\end{tikzpicture}}},
			\vcenter{\hbox{\begin{tikzpicture}[scale=0.4]
				\draw[thick] (0, 0.5)--(-1, 2);
				\draw[thick] (0, 0.5)--(1, 2);
				\draw[fill=white, thick] (0, 0.5) circle [radius=15pt];
				\draw[fill=white, thick] (-1, 2) circle [radius=15pt];
				\draw[fill=white, thick] (1, 2) circle [radius=15pt];
				\node at (0, 0.5) {$1$};
				\node at (-1, 2) {$3$};
				\node at (1, 2) {$2$};
			\end{tikzpicture}}},
			\vcenter{\hbox{\begin{tikzpicture}[scale=0.4]
				\draw[thick] (0, 0.5)--(-1, 2);
				\draw[thick] (0, 0.5)--(1, 2);
				\draw[fill=white, thick] (0, 0.5) circle [radius=15pt];
				\draw[fill=white, thick] (-1, 2) circle [radius=15pt];
				\draw[fill=white, thick] (1, 2) circle [radius=15pt];
				\node at (0, 0.5) {$2$};
				\node at (-1, 2) {$1$};
				\node at (1, 2) {$3$};
			\end{tikzpicture}}},
			\vcenter{\hbox{\begin{tikzpicture}[scale=0.4]
				\draw[thick] (0, 0.5)--(-1, 2);
				\draw[thick] (0, 0.5)--(1, 2);
				\draw[fill=white, thick] (0, 0.5) circle [radius=15pt];
				\draw[fill=white, thick] (-1, 2) circle [radius=15pt];
				\draw[fill=white, thick] (1, 2) circle [radius=15pt];
				\node at (0, 0.5) {$2$};
				\node at (-1, 2) {$3$};
				\node at (1, 2) {$1$};
			\end{tikzpicture}}},
			\vcenter{\hbox{\begin{tikzpicture}[scale=0.4]
				\draw[thick] (0, 0.5)--(-1, 2);
				\draw[thick] (0, 0.5)--(1, 2);
				\draw[fill=white, thick] (0, 0.5) circle [radius=15pt];
				\draw[fill=white, thick] (-1, 2) circle [radius=15pt];
				\draw[fill=white, thick] (1, 2) circle [radius=15pt];
				\node at (0, 0.5) {$3$};
				\node at (-1, 2) {$1$};
				\node at (1, 2) {$2$};
			\end{tikzpicture}}},
			\vcenter{\hbox{\begin{tikzpicture}[scale=0.4]
				\draw[thick] (0, 0.5)--(-1, 2);
				\draw[thick] (0, 0.5)--(1, 2);
				\draw[fill=white, thick] (0, 0.5) circle [radius=15pt];
				\draw[fill=white, thick] (-1, 2) circle [radius=15pt];
				\draw[fill=white, thick] (1, 2) circle [radius=15pt];
				\node at (0, 0.5) {$3$};
				\node at (-1, 2) {$2$};
				\node at (1, 2) {$1$};
			\end{tikzpicture}}}
		\right\}, \text{ etc...}
	\end{multline*}
\end{Example}

We may see that for $n\geq 2$, we have that $\mathrm{C}$ is isomorphic to $\mathfrak{S}_n$ endowed with its left action.

\begin{Definition}
	Let $\tau$ a rooted planar tree. 
	A vertex $a$ is \emph{above} another vertex $b$ if $b$ belongs to one of the children of $a$.
	A vertex $a$ is \emph{at the left of} another vertex $b$ if there is a vertex $v$ such that $a$ belongs to a child $\tau_a$ of $v$, $b$ belongs to a child $\tau_b$ of $v$ and $\tau_a<\tau_b$ (we recall that the children of a vertex are totally ordered).    
\end{Definition}

\begin{Example}
	In the following example, $\alpha$ is above $\delta$, and $\alpha$ is at the left of $\beta$.
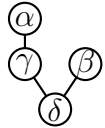
	\[\vcenter{\hbox{\begin{tikzpicture}[scale=0.4]
		\draw[thick] (0, 0.5)--(-1, 2);
		\draw[thick] (0, 0.5)--(1, 2);
		\draw[thick] (-1, 3.5)--(-1, 2);
		\draw[fill=white, thick] (0, 0.5) circle [radius=15pt];
		\draw[fill=white, thick] (-1, 2) circle [radius=15pt];
		\draw[fill=white, thick] (-1, 3.5) circle [radius=15pt];
		\draw[fill=white, thick] (1, 2) circle [radius=15pt];
		\node at (0, 0.5) {$\delta$};
		\node at (-1, 2) {$\gamma$};
		\node at (-1, 3.5) {$\alpha$};
		\node at (1, 2) {$\beta$};
	\end{tikzpicture}}}\]
\end{Example}

\begin{Definition}\label{def:dro}
	Let $\tau$ a rooted planar tree.
	The \emph{depth-right ordering} of the vertex of $\tau$ is obtained via the following algorithm:
	\begin{itemize}
		\item We start at the root, and we iteratively do the following steps until we have ordered all the vertices:
		\item If we are not at a leaf, we move to the root of the maximal (rightmost) child of the current vertex.
		\item If we are at a leaf, we remove it and move to its parent.
	\end{itemize}
	The order is the one given by the order of the first passage.
\end{Definition}

\begin{Example}\label{expl:rpt2tnc}
	In the rooted planar tree of the last example, the order is the following; it will be useful in the sequel to put the vertex on an axis according to the depth-right ordering:
	\[\vcenter
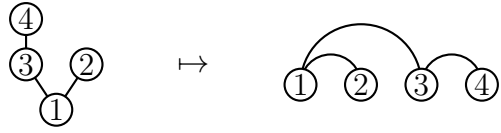
{\hbox{\begin{tikzpicture}[scale=0.4]
		\draw[thick] (0, 0.5)--(-1, 2);
		\draw[thick] (0, 0.5)--(1, 2);
		\draw[thick] (-1, 3.5)--(-1, 2);
		\draw[fill=white, thick] (0, 0.5) circle [radius=15pt];
		\draw[fill=white, thick] (-1, 2) circle [radius=15pt];
		\draw[fill=white, thick] (-1, 3.5) circle [radius=15pt];
		\draw[fill=white, thick] (1, 2) circle [radius=15pt];
		\node at (0, 0.5) {$1$};
		\node at (-1, 2) {$3$};
		\node at (-1, 3.5) {$4$};
		\node at (1, 2) {$2$};
	\end{tikzpicture}}}\qquad\mapsto\qquad
	\vcenter{\hbox{\begin{tikzpicture}[scale=0.4]
		\draw[thick] (0, 0)  (4, 0) arc(0:180:2) circle;
                \draw[thick] (4, 0)  (6, 0) arc(0:180:1) circle;
                \draw[thick] (0, 0)  (2, 0) arc(0:180:1) circle;
		\draw[fill=white, thick] (0, 0) circle [radius=15pt];
		\draw[fill=white, thick] (2, 0) circle [radius=15pt];
		\draw[fill=white, thick] (4, 0) circle [radius=15pt];
		\draw[fill=white, thick] (6, 0) circle [radius=15pt];
		\node at (0, 0) {$1$};
		\node at (2, 0) {$2$};
		\node at (4, 0) {$3$};
		\node at (6, 0) {$4$};
	\end{tikzpicture}}}
	\]
	In particular, with this geometric representation, we may notice that edges will never cross. 
\end{Example}

\begin{Definition}\label{def:l2ro}
	The \emph{left-to-right} ordering of the edges of a rooted planar tree is given by labelling each edge by its maximal vertex according to the depth-right ordering, and reversing the order. 
\end{Definition}

\begin{Example}
	Let us compute the left-to-right ordering for the following rooted planar tree:
	\[\vcenter
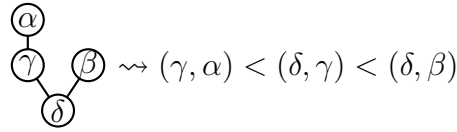
{\hbox{\begin{tikzpicture}[scale=0.4]
		\draw[thick] (0, 0.5)--(-1, 2);
		\draw[thick] (0, 0.5)--(1, 2);
		\draw[thick] (-1, 3.5)--(-1, 2);
		\draw[fill=white, thick] (0, 0.5) circle [radius=15pt];
		\draw[fill=white, thick] (-1, 2) circle [radius=15pt];
		\draw[fill=white, thick] (-1, 3.5) circle [radius=15pt];
		\draw[fill=white, thick] (1, 2) circle [radius=15pt];
		\node at (0, 0.5) {$\delta$};
		\node at (-1, 2) {$\gamma$};
		\node at (-1, 3.5) {$\alpha$};
		\node at (1, 2) {$\beta$};
	\end{tikzpicture}}}\rightsquigarrow(\gamma,\alpha)<(\delta,\gamma)<(\delta,\beta)\]
\end{Example}

\subsection{The Brace operad}

The Brace operad was first defined in \cite{GV95}, let us recall its definition and some of its basic properties.
Let $N=\{x_0,\dots,x_n\}$ with $n+1$ elements, we will use the following notation to write the operation encoded by a corolla:
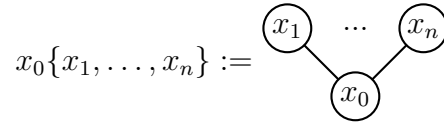
\[x_0\{x_1,\dots,x_n\}:=
	\vcenter{\hbox{\begin{tikzpicture}[scale=0.6]
		\draw[thick] (0, 0.5)--(-1.5, 2);
		\draw[thick] (0, 0.5)--(1.5, 2);
		\draw[fill=white, thick] (0, 0.5) circle [radius=15pt];
		\draw[fill=white, thick] (-1.5, 2) circle [radius=15pt];
		\draw[fill=white, thick] (1.5, 2) circle [radius=15pt];
		\node at (0, 0.5) {$x_0$};
		\node at (-1.5, 2) {$x_1$};
		\node at (1.5, 2) {$x_n$};
		\node at (0, 2) {$...$};
	\end{tikzpicture}}}
\]

\begin{Definition}
	Let $\mathcal{C}$ the linear species spanned by $\mathrm{C}$, and $\mathcal{T}(\mathcal{C})$ the free operad generated by $\mathcal{C}$.
	Let $\mathcal{R}$ the following set of relation:
	\begin{multline*}
		x_0\{x_1,\dots,x_m\}\{y_1,\dots,y_m\}=
		\sum_{0\leq i_1\leq j_1\leq\dots\leq i_n\leq j_n\leq m} x_0\{y_1,\dots,y_{i_1},x_1\{y_{i_1+1},\dots,y_{j_1}\},y_{j_1+1},\dots\\
		\dots,y_{i_n},x_n\{y_{i_n+1},\dots,y_{j_n}\},y_{j_n+1},\dots,y_m\}
	\end{multline*}
	Where we use the natural notation for the partial compositions of corollas with the convention $x\{\}=x$.
	The operad $\mathrm{Brace}$ is defined by generators and relations as:
	$$\mathrm{Brace}=\mathcal{T}(\mathcal{C})/\mathcal{R}.$$
\end{Definition}

\begin{Remark}
	The combinatorial interpretation of those relations is the following:
	\[
		\vcenter
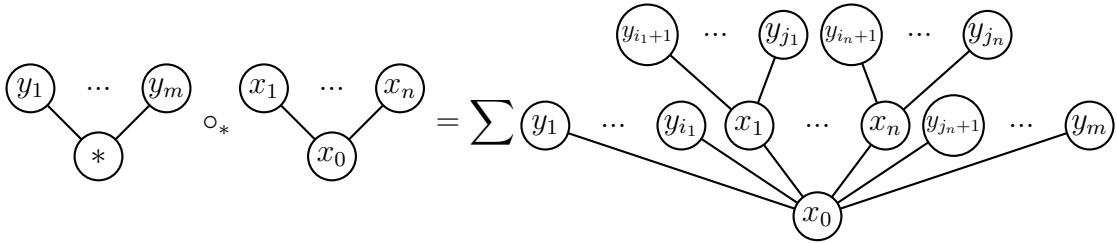
{\hbox{\begin{tikzpicture}[scale=0.6]
			\draw[thick] (0, 0.5)--(-1.5, 2);
			\draw[thick] (0, 0.5)--(1.5, 2);
			\draw[fill=white, thick] (0, 0.5) circle [radius=15pt];
			\draw[fill=white, thick] (-1.5, 2) circle [radius=15pt];
			\draw[fill=white, thick] (1.5, 2) circle [radius=15pt];
			\node at (0, 0.5) {$*$};
			\node at (-1.5, 2) {$y_1$};
			\node at (1.5, 2) {$y_m$};
			\node at (0, 2) {$...$};
		\end{tikzpicture}}}\circ_*
		\vcenter{\hbox{\begin{tikzpicture}[scale=0.6]
			\draw[thick] (0, 0.5)--(-1.5, 2);
			\draw[thick] (0, 0.5)--(1.5, 2);
			\draw[fill=white, thick] (0, 0.5) circle [radius=15pt];
			\draw[fill=white, thick] (-1.5, 2) circle [radius=15pt];
			\draw[fill=white, thick] (1.5, 2) circle [radius=15pt];
			\node at (0, 0.5) {$x_0$};
			\node at (-1.5, 2) {$x_1$};
			\node at (1.5, 2) {$x_n$};
			\node at (0, 2) {$...$};
		\end{tikzpicture}}}=\sum
		\vcenter{\hbox{\begin{tikzpicture}[scale=0.6]
			\draw[thick] (0, 0)--(-6, 2);
			\draw[thick] (0, 0)--(-3, 2);
			\draw[thick] (0, 0)--(-1.5, 2);
			\draw[thick] (-1.5, 2)--(-3.75, 4);
			\draw[thick] (-1.5, 2)--(-.75, 4);
			\draw[thick] (0, 0)--(1.5, 2);
			\draw[thick] (1.5, 2)--(.75, 4);
			\draw[thick] (1.5, 2)--(3.75, 4);
			\draw[thick] (0, 0)--(3, 2);
			\draw[thick] (0, 0)--(6, 2);
			\draw[fill=white, thick] (0, 0) circle [radius=15pt];
			\draw[fill=white, thick] (-6, 2) circle [radius=15pt];
			\draw[fill=white, thick] (-3, 2) circle [radius=15pt];
			\draw[fill=white, thick] (-1.5, 2) circle [radius=15pt];
			\draw[fill=white, thick] (-3.75, 4) circle [radius=19pt];
			\draw[fill=white, thick] (-.75, 4) circle [radius=15pt];
			\draw[fill=white, thick] (1.5, 2) circle [radius=15pt];
			\draw[fill=white, thick] (.75, 4) circle [radius=19pt];
			\draw[fill=white, thick] (3.75, 4) circle [radius=15pt];
			\draw[fill=white, thick] (3, 2) circle [radius=19pt];
			\draw[fill=white, thick] (6, 2) circle [radius=15pt];
			\node at (0, 0) {$x_0$};
			\node at (-6, 2) {$y_1$};
			\node at (-4.5, 2) {$...$};
			\node at (-3, 2) {$y_{i_1}$};
			\node at (-1.5, 2) {$x_1$};
			\node at (-3.75, 4) {\scalebox{0.8}{$y_{i_1+1}$}};
			\node at (-2.25, 4) {$...$};
			\node at (-0.75, 4) {$y_{j_1}$};
			\node at (0, 2) {$...$};
			\node at (1.5, 2) {$x_n$};
			\node at (0.75, 4) {\scalebox{0.8}{$y_{i_n+1}$}};
			\node at (2.25, 4) {$...$};
			\node at (3.75, 4) {$y_{j_n}$};
			\node at (3, 2) {\scalebox{0.8}{$y_{j_n+1}$}};
			\node at (4.5, 2) {$...$};
			\node at (6, 2) {$y_{m}$};
		\end{tikzpicture}}}
	\]
\end{Remark}

The combinatorial interpretation of this operad was given by Chapoton \cite{C02}.
We will not use the full combinatorial interpretation; however, we will use a combinatorial characterisation of $\mathrm{Brace}$.
First, let us point out that, since rooted planar corollas are rooted planar trees, we have an inclusion $\mathrm{C}\subseteq\mathrm{PT}$.

\begin{Theorem}[{Corollary of \cite[Section 2.2]{C02}}]\label{thm:carac}
	There is a unique structure of operad on $\mathcal{PT}$ such that:
	\begin{enumerate}
		\item the natural inclusion of $\mathrm{C}\subseteq\mathrm{PT}$ induces an isomorphism of operad with $\mathrm{Brace}$; and
		\item the partial composition of the rooted planar tree $\sigma$ in a leaf $i$ of the rooted planar tree $\tau$ is the rooted planar tree obtained by replacing $i$ by $\sigma$ in $\tau$.
	\end{enumerate}
\end{Theorem}

\begin{proof}
	The actual operad structure satisfying both properties is constructed in \cite[Section 2.1]{C02}.
	Let $\gamma,\gamma'$ two operadic structures on $\mathcal{PT}$ satisfying both properties.
	By Property 1, we have an isomorphism of operad $\varphi:(\mathcal{PT},\gamma)\to(\mathcal{PT},\gamma')$ such that $\varphi_{|\mathcal{C}}=\id_{\mathcal{C}}$.
	By Property 2, the partial composition induced by $\gamma$ and $\gamma'$ coincides when composing in a leaf.
	By definition of a rooted planar tree, any rooted planar tree is obtained by replacing the leaves of a corolla by rooted planar trees of lower height. 
	Thus, any planar rooted tree $\tau$ can be obtained by iteratively composing corollas in their leaves.
	Since $\varphi$ is a morphism of operad, $\varphi(\tau)=\tau$, and $\gamma=\gamma'$.
\end{proof}

\begin{Remark}
	The actual operadic structure on $\mathcal{PT}$ is very explicit and combinatorial.
	This is the planar version of the Chapoton-Livernet insertion of rooted (nonplanar) trees, see \cite{CL01}.
	However, we will not recall its precise definition in this article since the above characterisation is enough.
\end{Remark}	

The operad we will actually study is $\Lambda\mathrm{Brace}$ the operadic suspension of $\mathrm{Brace}$.
In $\Lambda\mathrm{Brace}$, the species $\mathcal{C}$ is shifted; indeed, the space $\mathcal{C}(n)$ is in degree $n-1$, and some Koszul signs appear in the generating relations of $\mathcal{R}$.
Hopefully, the combinatorial interpretation of $\mathrm{Brace}$ given by Chapoton allows us to compute this sign.
Indeed, we may notice that the operadic suspension corresponds to imposing that each edge of the rooted planar trees is of degree $1$.
Thus, to compute the Koszul signs, it suffices to fix an order of edges for each rooted planar tree, to compute the composition, and to compute the signature of the permutation between the order of edges appearing in the computation, and the fixed order.

\begin{Proposition}\label{prp:ks}
	By convention, let us order the edges of a partial composition by first ordering the edges of the left term, and then the right term. 
	With the left-to-right ordering of the edges, the defining relations of $\Lambda\mathrm{Brace}$ are:
	\[
		\vcenter{\hbox{\begin{tikzpicture}[scale=0.6]
			\draw[thick] (0, 0.5)--(-1.5, 2);
			\draw[thick] (0, 0.5)--(1.5, 2);
			\draw[fill=white, thick] (0, 0.5) circle [radius=15pt];
			\draw[fill=white, thick] (-1.5, 2) circle [radius=15pt];
			\draw[fill=white, thick] (1.5, 2) circle [radius=15pt];
			\node at (0, 0.5) {$*$};
			\node at (-1.5, 2) {$y_1$};
			\node at (1.5, 2) {$y_m$};
			\node at (0, 2) {$...$};
		\end{tikzpicture}}}\circ_*
		\vcenter{\hbox{\begin{tikzpicture}[scale=0.6]
			\draw[thick] (0, 0.5)--(-1.5, 2);
			\draw[thick] (0, 0.5)--(1.5, 2);
			\draw[fill=white, thick] (0, 0.5) circle [radius=15pt];
			\draw[fill=white, thick] (-1.5, 2) circle [radius=15pt];
			\draw[fill=white, thick] (1.5, 2) circle [radius=15pt];
			\node at (0, 0.5) {$x_0$};
			\node at (-1.5, 2) {$x_1$};
			\node at (1.5, 2) {$x_n$};
			\node at (0, 2) {$...$};
		\end{tikzpicture}}}=\sum (-1)^\varepsilon
		\vcenter{\hbox{\begin{tikzpicture}[scale=0.6]
			\draw[thick] (0, 0)--(-6, 2);
			\draw[thick] (0, 0)--(-3, 2);
			\draw[thick] (0, 0)--(-1.5, 2);
			\draw[thick] (-1.5, 2)--(-3.75, 4);
			\draw[thick] (-1.5, 2)--(-.75, 4);
			\draw[thick] (0, 0)--(1.5, 2);
			\draw[thick] (1.5, 2)--(.75, 4);
			\draw[thick] (1.5, 2)--(3.75, 4);
			\draw[thick] (0, 0)--(3, 2);
			\draw[thick] (0, 0)--(6, 2);
			\draw[fill=white, thick] (0, 0) circle [radius=15pt];
			\draw[fill=white, thick] (-6, 2) circle [radius=15pt];
			\draw[fill=white, thick] (-3, 2) circle [radius=15pt];
			\draw[fill=white, thick] (-1.5, 2) circle [radius=15pt];
			\draw[fill=white, thick] (-3.75, 4) circle [radius=19pt];
			\draw[fill=white, thick] (-.75, 4) circle [radius=15pt];
			\draw[fill=white, thick] (1.5, 2) circle [radius=15pt];
			\draw[fill=white, thick] (.75, 4) circle [radius=19pt];
			\draw[fill=white, thick] (3.75, 4) circle [radius=15pt];
			\draw[fill=white, thick] (3, 2) circle [radius=19pt];
			\draw[fill=white, thick] (6, 2) circle [radius=15pt];
			\node at (0, 0) {$x_0$};
			\node at (-6, 2) {$y_1$};
			\node at (-4.5, 2) {$...$};
			\node at (-3, 2) {$y_{i_1}$};
			\node at (-1.5, 2) {$x_1$};
			\node at (-3.75, 4) {\scalebox{0.8}{$y_{i_1+1}$}};
			\node at (-2.25, 4) {$...$};
			\node at (-0.75, 4) {$y_{j_1}$};
			\node at (0, 2) {$...$};
			\node at (1.5, 2) {$x_n$};
			\node at (0.75, 4) {\scalebox{0.8}{$y_{i_n+1}$}};
			\node at (2.25, 4) {$...$};
			\node at (3.75, 4) {$y_{j_n}$};
			\node at (3, 2) {\scalebox{0.8}{$y_{j_n+1}$}};
			\node at (4.5, 2) {$...$};
			\node at (6, 2) {$y_{m}$};
		\end{tikzpicture}}}
	\]
	with:
	$$\varepsilon = (j_n-j_{n-1}) + 2(j_{n-1}-j_{n-2}) + \dots + (n-1)(j_2 - j_1) +  n(j_1 - 0)\equiv \sum_{\ell=1}^n j_\ell$$
\end{Proposition}

\begin{proof}
	Let's apply the strategy we discussed in the previous paragraph. 
	The edges attached to the vertices $y_{j_{n-1}+1}$ up to $y_{j_n}$ are now before the edge attached to $x_n$, the edges attached to the vertices $y_{j_{n-2}+1}$ up to $y_{i_{n-1}}$ are now before the edges attached to $x_{n-1}$ and $x_n$, and so on ...
\end{proof}

\begin{Remark}
	We have a lot of possible choices to order the edges.
	We have chosen this specific one for two reasons.
	On a first hand, the simplicity of the formula we get.
	On a second edge, it will be nicely behaving with the noncrossing 2-partitions of the next section.
\end{Remark}

\section{Noncrossing 2-Partition}

\subsection{Posets and poset cohomology}

Let us recall the definition of noncrossing 2-partitions as defined in \cite{E80}, and the operadic structure of their cohomology defined in \cite{DD25}.

\begin{Definition}
	Let $n\in \mathbb{N}$, and $\rho\vdash\{1,\dots,n\}$ a partition of $\{1,\dots,n\}$.
	A \emph{crossing} of $\rho$ is a quadruple $(i_1<i_2<i_3<i_4)$ such that we have $q_1\neq q_2\in\rho$ with $i_1,i_3\in q_1$, and $i_2,i_4\in q_2$.
	The partition $\rho$ is a \emph{noncrossing partition of $n$} if it has no crossings.\\
	Let $N$ a finite set. 
	A \emph{noncrossing 2-partition of $N$} is the data of:
	\begin{itemize}
		\item a noncrossing partition $\rho$ of $n=|N|$;
		\item a partition $\pi$ of $N$;
		\item a bijection $f:\pi\to\rho$ such that $|f(p)|=|p|$ for any $p\in\pi$.
	\end{itemize}
\end{Definition}

Let us now define the partial order on noncrossing 2-partitions.

\begin{Definition}
	We recall that we have a partial ordering on partitions.
	Let $\pi,\pi'\vdash N$, we have $\pi\geq\pi'$ if for all $p\in\pi$ we have $p'\in\pi$ such that $p\subseteq p'$.
	We denote by $\Pi(N)$ the partition poset of $N$, we may notice that $\Pi$ is a functor $\mathbb{B}\to \mathrm{PoSet}$ with $\mathrm{PoSet}$ the category of posets.\\
	Let $(\rho,\pi,f)$ and $(\rho',\pi',f')$ noncrossing 2-partitions of $N$.
	Then $(\rho,\pi,f)\geq(\rho',\pi',f')$ if $\rho\geq\rho'$, $\pi\geq \pi'$, and for $p\in\pi$, $p'\in \pi$, we have $p\subseteq p'\Rightarrow f(p)\subseteq f'(p')$.
	Let $\Pi_2(N)$ the poset defined by the noncrossing 2-partition of $N$.
	We may notice that $\Pi_2$ is also a functor $\mathbb{B}\to \mathrm{PoSet}$.
\end{Definition}

\begin{Definition}
	We define $h^\bullet(\mathrm{P})$ with $\mathrm{P}$ a poset, as the cohomology of $\mathrm{P}$ relative to its extremal elements.
	More explicitly, this is the homology of the chain complex $c^\bullet(\mathrm{P})$ where $c^k(\mathrm{P})$ is the span of chains $x_0<\dots<x_k$ of length $k$ of $\mathrm{P}$ with $x_0$ minimal and $x_k$ maximal, and the differential is the natural insertion of an element of $\mathrm{P}$ in the chain:
	$$d(x_0<\dots<x_k)=\sum_{i=1}^k (-1)^i \sum_{x_{i-1}<y<x_i} x_0<\dots<x_{i-1}<y<x_i<\dots<x_k$$
\end{Definition}

Because $\Pi_2$ is a functor $\mathbb{B}\to \mathrm{PoSet}$, we have that $h^\bullet(\Pi_2)$ is a functor $\mathbb{B}\to \mathrm{gr}\mathrm{Vect}$, and thus, a graded linear species.
Let us recall the following result on $h^\bullet(\Pi_2(N))$:

\begin{Proposition}[{\cite[Proposition 5.8]{DD25}}]
	Let $N$ a finite set with $n$ elements, then $h^\bullet(\Pi_2(N))$ is concentrated in degree $n-1$.
	Moreover, let $\Lambda h^\bullet(\Pi_2(N))$ the $n-1$ shift of $h^\bullet(\Pi_2(N))$, then the graded linear species $\Lambda h^\bullet(\Pi_2)$ is concentrated in degree $0$ and is isomorphic to the species $\mathcal{PT}$.
\end{Proposition}

\begin{Definition}
	Let $N$ a finite set.
	Let $a_N:\Pi_2(N)\to\Pi(N); (\pi,\rho,f)\mapsto \pi$, this is a natural transformation between the functors $\Pi_2$ and $\Pi$.
	For $x=(\pi,\rho,f)\in \Pi_2(N)$, let: 
	$$\psi_x:\left(\Pi_2\right)_{\geq x}\to\prod_{q\in\pi}\Pi_2(q);(\beta,\mu,g)\mapsto((\beta_{|q},\mu_{|q},g_{|q}))_{q\in\pi}$$
	where we denote by $\beta_{|q},\mu_{|q},g_{|q}$ the natural restriction of $\beta,\mu,g$ to $q\subseteq N$.
	Let $\sigma_\rho:\rho\to\{1,\dots,|\rho|\}$ the ordering of $\rho$ according to the minimal element of each block, meaning that for $q_1,q_2\in \rho$, we have $\sigma_\rho(q_1)>\sigma_\rho(q_2)$ if and only if $\min(q_1)>\min(q_2)$.
	Thus, for $\mu\leq \rho$ a noncrossing partition of $\{1,\dots,n\}$ we can see $\mu$ as a partition $\varphi_\rho(\mu)$ of $\rho$, moreover via this identification $\sigma_\rho:\rho\to\{1,\dots,|\rho|\}$, we have that $\varphi_\rho(\mu)$ is a noncrossing partition of $\{1,\dots,|\rho|\}$.
	Thus, we have a map:
	$$\varphi_x:\left(\Pi_2\right)_{\leq x}(N)\to\Pi_2(\pi);(\beta,\mu,g)\mapsto (\alpha,\varphi_\rho,h)$$
	with $\alpha$ the partition of $\pi$ induced by $\beta$, and $h$ the isomorphism $\pi\to\{1,\dots,|\rho|\}$ induced by $f$.
\end{Definition}

This is the \emph{operadic partition structure} of $\Pi_2$ as defined in \cite{DD25}. 
We refer to \cite{DD25} for the actual definition of an operadic partition structure, and for the definition of the operadic structure on the cohomology of an operadic partition poset, in particular for the noncrossing 2-partitions.
The definition we give is the minimal one required to carry out actual computations, which is enough for this article.
We refer to \cite{DD25} for the general theory of operadic partition posets, the general definition of the operadic structure induced on the cohomology, and more insight and motivation for the definition given below.

\begin{Definition}
	Let $A,B$ two disjoint finite sets, and $\pi=\{\{a\}\mid a\in A\}\sqcup\{B\}\vdash A\sqcup B$.
	We have that:
	$$\Pi_2(\pi)\simeq \Pi_2(A\sqcup\{*\}),\quad\text{and}\quad\prod_{q\in\pi}\Pi_2(q)\simeq\Pi_2(B)$$
	Let $x\in a_{A\sqcup B}^{-1}(\pi)$, and 
	$$\varphi_x^*:c^\bullet(\Pi_2(A\sqcup\{*\}))\to c^\bullet(\left(\Pi_2\right)_{\leq x}(A\sqcup B)),\quad\text{and}\quad
	\psi_x^*:c^\bullet(\Pi_2(B))\to c^\bullet(\left(\Pi_2\right)_{\geq x}(A\sqcup B))$$
	the maps induced by $\varphi_x, \psi_x$ on the cochains.
	We get partial composition maps by setting:
	$$\circ_* : c^\bullet(\Pi_2(A\sqcup\{*\}))\otimes c^\bullet(\Pi_2(B))\to\bigoplus_{x\in a_{A\sqcup B}^{-1}(\pi)} c^\bullet(\left(\Pi_2\right)_{\leq x}(A\sqcup B))\otimes c^\bullet(\left(\Pi_2\right)_{\geq x}(A\sqcup B))\to c^\bullet(\Pi_2(A\sqcup B)$$ 
\end{Definition}

By \cite[Theorem 2.8]{DD25}, the partial compositions above defined induce a structure of graded operad on $h^\bullet(\Pi_2)$.
Those partial compositions do not directly induce an operadic structure at the level of cochains since associativity only holds ``up to homotopy''.

\begin{Example}
	Let us copy the first part of the proof of \cite[Proposition 5.9]{DD25} to get an example of computation.
	Let $A=\{\gamma\}$, $B=\{\alpha,\beta\}$, and let us compute:
	\[\left(\vcenter{\hbox{\begin{tikzpicture}[scale=0.4]
		\draw[thick] (0, 0)  (2, 0) arc(0:180:1) circle;
		\draw[fill=white, thick] (0, 0) circle [radius=15pt];
		\draw[fill=white, thick] (2, 0) circle [radius=15pt];
		\node at (0, 0) {$*$};
		\node at (2, 0) {$\gamma$};
	\end{tikzpicture}}}<
	\vcenter{\hbox{\begin{tikzpicture}[scale=0.4]
		\draw[fill=white, thick] (0, 0) circle [radius=15pt];
		\draw[fill=white, thick] (2, 0) circle [radius=15pt];
		\node at (0, 0) {$*$};
		\node at (2, 0) {$\gamma$};
	\end{tikzpicture}}}\right) \circ_*
	\left(\vcenter{\hbox{\begin{tikzpicture}[scale=0.4]
		\draw[thick] (0, 0)  (2, 0) arc(0:180:1) circle;
		\draw[fill=white, thick] (0, 0) circle [radius=15pt];
		\draw[fill=white, thick] (2, 0) circle [radius=15pt];
		\node at (0, 0) {$\alpha$};
		\node at (2, 0) {$\beta$};
	\end{tikzpicture}}}<
	\vcenter{\hbox{\begin{tikzpicture}[scale=0.4]
		\draw[fill=white, thick] (0, 0) circle [radius=15pt];
		\draw[fill=white, thick] (2, 0) circle [radius=15pt];
		\node at (0, 0) {$\alpha$};
		\node at (2, 0) {$\beta$};
	\end{tikzpicture}}}\right)
	\]
	The partition $\pi$ of the previous definition is $\{\{\alpha,\beta\},\{\gamma\}\}$.
	Thus, we have:
	\[a^{-1}(\pi)=\left\{
	\vcenter{\hbox{\begin{tikzpicture}[scale=0.4]
		\draw[draw=none, thick] (0, 0)  (4, 0) arc(0:180:2) circle;
		\draw[thick] (0, 0)  (2, 0) arc(0:180:1) circle;
		\draw[fill=white, thick] (0, 0) circle [radius=15pt];
		\draw[fill=white, thick] (2, 0) circle [radius=15pt];
		\draw[fill=white, thick] (4, 0) circle [radius=15pt];
		\node at (0, 0) {$\alpha$};
		\node at (2, 0) {$\beta$};
		\node at (4, 0) {$\gamma$};
	\end{tikzpicture}}}\;,\;
	\vcenter{\hbox{\begin{tikzpicture}[scale=0.4]
		\draw[thick] (0, 0)  (4, 0) arc(0:180:2) circle;
		\draw[fill=white, thick] (0, 0) circle [radius=15pt];
		\draw[fill=white, thick] (2, 0) circle [radius=15pt];
		\draw[fill=white, thick] (4, 0) circle [radius=15pt];
		\node at (0, 0) {$\alpha$};
		\node at (2, 0) {$\gamma$};
		\node at (4, 0) {$\beta$};
	\end{tikzpicture}}}\;,\;
	\vcenter{\hbox{\begin{tikzpicture}[scale=0.4]
		\draw[draw=none, thick] (0, 0)  (4, 0) arc(0:180:2) circle;
		\draw[thick] (2, 0)  (4, 0) arc(0:180:1) circle;
		\draw[fill=white, thick] (0, 0) circle [radius=15pt];
		\draw[fill=white, thick] (2, 0) circle [radius=15pt];
		\draw[fill=white, thick] (4, 0) circle [radius=15pt];
		\node at (0, 0) {$\gamma$};
		\node at (2, 0) {$\alpha$};
		\node at (4, 0) {$\beta$};
	\end{tikzpicture}}}\right\}
	\]
	Let us denote those respectively $x$, $y$, and $z$.
	Let us point out that, since the elements inside the partitions are not ordered, we have that:
	\[\vcenter{\hbox{\begin{tikzpicture}[scale=0.4]
		\draw[thick] (0, 0)  (2, 0) arc(0:180:1) circle;
		\draw[fill=white, thick] (0, 0) circle [radius=15pt];
		\draw[fill=white, thick] (2, 0) circle [radius=15pt];
		\draw[fill=white, thick] (4, 0) circle [radius=15pt];
		\node at (0, 0) {$\alpha$};
		\node at (2, 0) {$\beta$};
		\node at (4, 0) {$\gamma$};
	\end{tikzpicture}}}\;=\;
	\vcenter{\hbox{\begin{tikzpicture}[scale=0.4]
		\draw[thick] (0, 0)  (2, 0) arc(0:180:1) circle;
		\draw[fill=white, thick] (0, 0) circle [radius=15pt];
		\draw[fill=white, thick] (2, 0) circle [radius=15pt];
		\draw[fill=white, thick] (4, 0) circle [radius=15pt];
		\node at (0, 0) {$\beta$};
		\node at (2, 0) {$\alpha$};
		\node at (4, 0) {$\gamma$};
	\end{tikzpicture}}}\]
	We have:
	\begin{align*}
	\varphi^*_x\left(\vcenter{\hbox{\begin{tikzpicture}[scale=0.4]
		\draw[thick] (0, 0)  (2, 0) arc(0:180:1) circle;
		\draw[fill=white, thick] (0, 0) circle [radius=15pt];
		\draw[fill=white, thick] (2, 0) circle [radius=15pt];
		\node at (0, 0) {$*$};
		\node at (2, 0) {$\gamma$};
	\end{tikzpicture}}}<
	\vcenter{\hbox{\begin{tikzpicture}[scale=0.4]
		\draw[fill=white, thick] (0, 0) circle [radius=15pt];
		\draw[fill=white, thick] (2, 0) circle [radius=15pt];
		\node at (0, 0) {$*$};
		\node at (2, 0) {$\gamma$};
	\end{tikzpicture}}}\right) &=
	\left(\vcenter{\hbox{\begin{tikzpicture}[scale=0.4]
		\draw[thick] (0, 0)  (2, 0) arc(0:180:1) circle;
		\draw[thick] (2, 0)  (4, 0) arc(0:180:1) circle;
		\draw[fill=white, thick] (0, 0) circle [radius=15pt];
		\draw[fill=white, thick] (2, 0) circle [radius=15pt];
		\draw[fill=white, thick] (4, 0) circle [radius=15pt];
		\node at (0, 0) {$\alpha$};
		\node at (2, 0) {$\beta$};
		\node at (4, 0) {$\gamma$};
	\end{tikzpicture}}}<
	\vcenter{\hbox{\begin{tikzpicture}[scale=0.4]
		\draw[thick] (0, 0)  (2, 0) arc(0:180:1) circle;
		\draw[fill=white, thick] (0, 0) circle [radius=15pt];
		\draw[fill=white, thick] (2, 0) circle [radius=15pt];
		\draw[fill=white, thick] (4, 0) circle [radius=15pt];
		\node at (0, 0) {$\alpha$};
		\node at (2, 0) {$\beta$};
		\node at (4, 0) {$\gamma$};
	\end{tikzpicture}}}\right)\\ 
	\varphi^*_y\left(\vcenter{\hbox{\begin{tikzpicture}[scale=0.4]
		\draw[thick] (0, 0)  (2, 0) arc(0:180:1) circle;
		\draw[fill=white, thick] (0, 0) circle [radius=15pt];
		\draw[fill=white, thick] (2, 0) circle [radius=15pt];
		\node at (0, 0) {$*$};
		\node at (2, 0) {$\gamma$};
	\end{tikzpicture}}}<
	\vcenter{\hbox{\begin{tikzpicture}[scale=0.4]
		\draw[fill=white, thick] (0, 0) circle [radius=15pt];
		\draw[fill=white, thick] (2, 0) circle [radius=15pt];
		\node at (0, 0) {$*$};
		\node at (2, 0) {$\gamma$};
	\end{tikzpicture}}}\right) &=
	\left(\vcenter{\hbox{\begin{tikzpicture}[scale=0.4]
		\draw[draw=none, thick] (0, 0)  (4, 0) arc(0:180:2) circle;
		\draw[thick] (0, 0)  (2, 0) arc(0:180:1) circle;
		\draw[thick] (2, 0)  (4, 0) arc(0:180:1) circle;
		\draw[fill=white, thick] (0, 0) circle [radius=15pt];
		\draw[fill=white, thick] (2, 0) circle [radius=15pt];
		\draw[fill=white, thick] (4, 0) circle [radius=15pt];
		\node at (0, 0) {$\alpha$};
		\node at (2, 0) {$\beta$};
		\node at (4, 0) {$\gamma$};
	\end{tikzpicture}}}<
	\vcenter{\hbox{\begin{tikzpicture}[scale=0.4]
		\draw[thick] (0, 0)  (4, 0) arc(0:180:2) circle;
		\draw[fill=white, thick] (0, 0) circle [radius=15pt];
		\draw[fill=white, thick] (2, 0) circle [radius=15pt];
		\draw[fill=white, thick] (4, 0) circle [radius=15pt];
		\node at (0, 0) {$\alpha$};
		\node at (2, 0) {$\gamma$};
		\node at (4, 0) {$\beta$};
	\end{tikzpicture}}}\right)\\ 
	\varphi^*_z\left(\vcenter{\hbox{\begin{tikzpicture}[scale=0.4]
		\draw[thick] (0, 0)  (2, 0) arc(0:180:1) circle;
		\draw[fill=white, thick] (0, 0) circle [radius=15pt];
		\draw[fill=white, thick] (2, 0) circle [radius=15pt];
		\node at (0, 0) {$*$};
		\node at (2, 0) {$\gamma$};
	\end{tikzpicture}}}<
	\vcenter{\hbox{\begin{tikzpicture}[scale=0.4]
		\draw[fill=white, thick] (0, 0) circle [radius=15pt];
		\draw[fill=white, thick] (2, 0) circle [radius=15pt];
		\node at (0, 0) {$*$};
		\node at (2, 0) {$\gamma$};
	\end{tikzpicture}}}\right) &=0
	\end{align*}
	Similarly, let us compute the $\psi^*$ part.
	We have:
	\begin{align*}
	\psi^*_x\left(\left(\vcenter{\hbox{\begin{tikzpicture}[scale=0.4]
		\draw[thick] (0, 0)  (2, 0) arc(0:180:1) circle;
		\draw[fill=white, thick] (0, 0) circle [radius=15pt];
		\draw[fill=white, thick] (2, 0) circle [radius=15pt];
		\node at (0, 0) {$\alpha$};
		\node at (2, 0) {$\beta$};
	\end{tikzpicture}}}<
	\vcenter{\hbox{\begin{tikzpicture}[scale=0.4]
		\draw[fill=white, thick] (0, 0) circle [radius=15pt];
		\draw[fill=white, thick] (2, 0) circle [radius=15pt];
		\node at (0, 0) {$\alpha$};
		\node at (2, 0) {$\beta$};
	\end{tikzpicture}}}\right)\otimes 
	\vcenter{\hbox{\begin{tikzpicture}[scale=0.4]
		\draw[fill=white, thick] (0, 0) circle [radius=15pt];
		\node at (0, 0) {$\gamma$};
	\end{tikzpicture}}}\right) &=
	\left(\vcenter{\hbox{\begin{tikzpicture}[scale=0.4]
		\draw[thick] (0, 0)  (2, 0) arc(0:180:1) circle;
		\draw[fill=white, thick] (0, 0) circle [radius=15pt];
		\draw[fill=white, thick] (2, 0) circle [radius=15pt];
		\draw[fill=white, thick] (4, 0) circle [radius=15pt];
		\node at (0, 0) {$\alpha$};
		\node at (2, 0) {$\beta$};
		\node at (4, 0) {$\gamma$};
	\end{tikzpicture}}}<
	\vcenter{\hbox{\begin{tikzpicture}[scale=0.4]
		\draw[draw=none, thick] (0, 0)  (2, 0) arc(0:180:1) circle;
		\draw[fill=white, thick] (0, 0) circle [radius=15pt];
		\draw[fill=white, thick] (2, 0) circle [radius=15pt];
		\draw[fill=white, thick] (4, 0) circle [radius=15pt];
		\node at (0, 0) {$\alpha$};
		\node at (2, 0) {$\beta$};
		\node at (4, 0) {$\gamma$};
	\end{tikzpicture}}}\right)\\ 
	\psi^*_y\left(\left(\vcenter{\hbox{\begin{tikzpicture}[scale=0.4]
		\draw[thick] (0, 0)  (2, 0) arc(0:180:1) circle;
		\draw[fill=white, thick] (0, 0) circle [radius=15pt];
		\draw[fill=white, thick] (2, 0) circle [radius=15pt];
		\node at (0, 0) {$\alpha$};
		\node at (2, 0) {$\beta$};
	\end{tikzpicture}}}<
	\vcenter{\hbox{\begin{tikzpicture}[scale=0.4]
		\draw[fill=white, thick] (0, 0) circle [radius=15pt];
		\draw[fill=white, thick] (2, 0) circle [radius=15pt];
		\node at (0, 0) {$\alpha$};
		\node at (2, 0) {$\beta$};
	\end{tikzpicture}}}\right)\otimes 
	\vcenter{\hbox{\begin{tikzpicture}[scale=0.4]
		\draw[fill=white, thick] (0, 0) circle [radius=15pt];
		\node at (0, 0) {$\gamma$};
	\end{tikzpicture}}}\right) &=
	\left(\vcenter{\hbox{\begin{tikzpicture}[scale=0.4]
		\draw[thick] (0, 0)  (4, 0) arc(0:180:2) circle;
		\draw[fill=white, thick] (0, 0) circle [radius=15pt];
		\draw[fill=white, thick] (2, 0) circle [radius=15pt];
		\draw[fill=white, thick] (4, 0) circle [radius=15pt];
		\node at (0, 0) {$\alpha$};
		\node at (2, 0) {$\gamma$};
		\node at (4, 0) {$\beta$};
	\end{tikzpicture}}}<
	\vcenter{\hbox{\begin{tikzpicture}[scale=0.4]
		\draw[draw=none, thick] (0, 0)  (4, 0) arc(0:180:2) circle;
		\draw[fill=white, thick] (0, 0) circle [radius=15pt];
		\draw[fill=white, thick] (2, 0) circle [radius=15pt];
		\draw[fill=white, thick] (4, 0) circle [radius=15pt];
		\node at (0, 0) {$\alpha$};
		\node at (2, 0) {$\gamma$};
		\node at (4, 0) {$\beta$};
	\end{tikzpicture}}}\right)\\ 
	\psi^*_z\left(\left(\vcenter{\hbox{\begin{tikzpicture}[scale=0.4]
		\draw[thick] (0, 0)  (2, 0) arc(0:180:1) circle;
		\draw[fill=white, thick] (0, 0) circle [radius=15pt];
		\draw[fill=white, thick] (2, 0) circle [radius=15pt];
		\node at (0, 0) {$\alpha$};
		\node at (2, 0) {$\beta$};
	\end{tikzpicture}}}<
	\vcenter{\hbox{\begin{tikzpicture}[scale=0.4]
		\draw[fill=white, thick] (0, 0) circle [radius=15pt];
		\draw[fill=white, thick] (2, 0) circle [radius=15pt];
		\node at (0, 0) {$\alpha$};
		\node at (2, 0) {$\beta$};
	\end{tikzpicture}}}\right)\otimes 
	\vcenter{\hbox{\begin{tikzpicture}[scale=0.4]
		\draw[fill=white, thick] (0, 0) circle [radius=15pt];
		\node at (0, 0) {$\gamma$};
	\end{tikzpicture}}}\right) &=
	\left(\vcenter{\hbox{\begin{tikzpicture}[scale=0.4]
		\draw[thick] (2, 0)  (4, 0) arc(0:180:1) circle;
		\draw[fill=white, thick] (0, 0) circle [radius=15pt];
		\draw[fill=white, thick] (2, 0) circle [radius=15pt];
		\draw[fill=white, thick] (4, 0) circle [radius=15pt];
		\node at (0, 0) {$\gamma$};
		\node at (2, 0) {$\alpha$};
		\node at (4, 0) {$\beta$};
	\end{tikzpicture}}}<
	\vcenter{\hbox{\begin{tikzpicture}[scale=0.4]
		\draw[draw=none, thick] (2, 0)  (4, 0) arc(0:180:1) circle;
		\draw[fill=white, thick] (0, 0) circle [radius=15pt];
		\draw[fill=white, thick] (2, 0) circle [radius=15pt];
		\draw[fill=white, thick] (4, 0) circle [radius=15pt];
		\node at (0, 0) {$\gamma$};
		\node at (2, 0) {$\alpha$};
		\node at (4, 0) {$\beta$};
	\end{tikzpicture}}}\right) 
	\end{align*}
	To finish the computation, we need to concatenate the chains given by $\varphi^*$ and $\psi^*$, and sum over $a^{-1}(\pi)$.
	\begin{multline*}
	\left(\vcenter{\hbox{\begin{tikzpicture}[scale=0.4]
		\draw[thick] (0, 0)  (2, 0) arc(0:180:1) circle;
		\draw[fill=white, thick] (0, 0) circle [radius=15pt];
		\draw[fill=white, thick] (2, 0) circle [radius=15pt];
		\node at (0, 0) {$*$};
		\node at (2, 0) {$\gamma$};
	\end{tikzpicture}}}<
	\vcenter{\hbox{\begin{tikzpicture}[scale=0.4]
		\draw[fill=white, thick] (0, 0) circle [radius=15pt];
		\draw[fill=white, thick] (2, 0) circle [radius=15pt];
		\node at (0, 0) {$*$};
		\node at (2, 0) {$\gamma$};
	\end{tikzpicture}}}\right) \circ_*
	\left(\vcenter{\hbox{\begin{tikzpicture}[scale=0.4]
		\draw[thick] (0, 0)  (2, 0) arc(0:180:1) circle;
		\draw[fill=white, thick] (0, 0) circle [radius=15pt];
		\draw[fill=white, thick] (2, 0) circle [radius=15pt];
		\node at (0, 0) {$\alpha$};
		\node at (2, 0) {$\beta$};
	\end{tikzpicture}}}<
	\vcenter{\hbox{\begin{tikzpicture}[scale=0.4]
		\draw[fill=white, thick] (0, 0) circle [radius=15pt];
		\draw[fill=white, thick] (2, 0) circle [radius=15pt];
		\node at (0, 0) {$\alpha$};
		\node at (2, 0) {$\beta$};
	\end{tikzpicture}}}\right)=\\
	\left(\vcenter{\hbox{\begin{tikzpicture}[scale=0.4]
		\draw[thick] (0, 0)  (2, 0) arc(0:180:1) circle;
		\draw[thick] (0, 0)  (4, 0) arc(0:180:1) circle;
		\draw[fill=white, thick] (0, 0) circle [radius=15pt];
		\draw[fill=white, thick] (2, 0) circle [radius=15pt];
		\draw[fill=white, thick] (4, 0) circle [radius=15pt];
		\node at (0, 0) {$\alpha$};
		\node at (2, 0) {$\beta$};
		\node at (4, 0) {$\gamma$};
	\end{tikzpicture}}}<
	\vcenter{\hbox{\begin{tikzpicture}[scale=0.4]
		\draw[thick] (0, 0)  (2, 0) arc(0:180:1) circle;
		\draw[fill=white, thick] (0, 0) circle [radius=15pt];
		\draw[fill=white, thick] (2, 0) circle [radius=15pt];
		\draw[fill=white, thick] (4, 0) circle [radius=15pt];
		\node at (0, 0) {$\alpha$};
		\node at (2, 0) {$\beta$};
		\node at (4, 0) {$\gamma$};
	\end{tikzpicture}}}<
	\vcenter{\hbox{\begin{tikzpicture}[scale=0.4]
		\draw[draw=none, thick] (0, 0)  (2, 0) arc(0:180:1) circle;
		\draw[fill=white, thick] (0, 0) circle [radius=15pt];
		\draw[fill=white, thick] (2, 0) circle [radius=15pt];
		\draw[fill=white, thick] (4, 0) circle [radius=15pt];
		\node at (0, 0) {$\alpha$};
		\node at (2, 0) {$\beta$};
		\node at (4, 0) {$\gamma$};
	\end{tikzpicture}}}\right)+\\ 
	\left(\vcenter{\hbox{\begin{tikzpicture}[scale=0.4]
		\draw[draw=none, thick] (0, 0)  (4, 0) arc(0:180:2) circle;
		\draw[thick] (0, 0)  (2, 0) arc(0:180:1) circle;
		\draw[thick] (0, 0)  (4, 0) arc(0:180:1) circle;
		\draw[fill=white, thick] (0, 0) circle [radius=15pt];
		\draw[fill=white, thick] (2, 0) circle [radius=15pt];
		\draw[fill=white, thick] (4, 0) circle [radius=15pt];
		\node at (0, 0) {$\alpha$};
		\node at (2, 0) {$\gamma$};
		\node at (4, 0) {$\beta$};
	\end{tikzpicture}}}<
	\vcenter{\hbox{\begin{tikzpicture}[scale=0.4]
		\draw[thick] (0, 0)  (4, 0) arc(0:180:2) circle;
		\draw[fill=white, thick] (0, 0) circle [radius=15pt];
		\draw[fill=white, thick] (2, 0) circle [radius=15pt];
		\draw[fill=white, thick] (4, 0) circle [radius=15pt];
		\node at (0, 0) {$\alpha$};
		\node at (2, 0) {$\gamma$};
		\node at (4, 0) {$\beta$};
	\end{tikzpicture}}}<
	\vcenter{\hbox{\begin{tikzpicture}[scale=0.4]
		\draw[draw=none, thick] (0, 0)  (4, 0) arc(0:180:2) circle;
		\draw[fill=white, thick] (0, 0) circle [radius=15pt];
		\draw[fill=white, thick] (2, 0) circle [radius=15pt];
		\draw[fill=white, thick] (4, 0) circle [radius=15pt];
		\node at (0, 0) {$\alpha$};
		\node at (2, 0) {$\gamma$};
		\node at (4, 0) {$\beta$};
	\end{tikzpicture}}}\right)		
	\end{multline*}

\end{Example}

We should notice the following important direct consequences of the definition of the partial composition $p_1\circ_* p_2$:
\begin{itemize}
	\item Each element $\lambda\in a^{-1}(\pi)$ gives rise to either one or zero terms.
	\item In each term of the partial composition, the elements of $A$ are in the same relative order as in $p_1$, and the elements of $B$ are in the same relative order as in $p_2$.
	\item In each term of the partial composition, the minimal element of $B$ in $p_2$ gets the same relative position as $*$ in $p_1$. 
\end{itemize}
Let us point out each of those assertions in the above example.
The first one is clear, for the second one, we should notice that we have $\alpha<\beta$ in each of the terms, and for the last one, we have $\alpha<\gamma$ in each of the terms.

\subsection{Complete Paths}

Because the cohomology of noncrossing 2-partitions is concentrated in maximal degrees, we will be mainly interested in $h^{n-1}(\Pi_2(N))$, and $c^{n-1}(\Pi_2(N))$.
Let us introduce some notations.
Let $(x_0<...<x_n)\in c^{n-1}(\Pi_2(N))$, because this chain is of maximal length, each transition $x_{i-1}<x_i$ is given by the fusion of two blocks of the underlying partitions.
Thus, we will represent those chains in one line by the consecutive fusions of blocks, by linking the minimum of each block.
Thus, we get:
\[
	\left(\vcenter{\hbox{\begin{tikzpicture}[scale=0.4]
		\draw[thick] (0, 0)  (2, 0) arc(0:180:1) circle;
		\draw[thick] (0, 0)  (4, 0) arc(0:180:1) circle;
		\draw[fill=white, thick] (0, 0) circle [radius=15pt];
		\draw[fill=white, thick] (2, 0) circle [radius=15pt];
		\draw[fill=white, thick] (4, 0) circle [radius=15pt];
		\node at (0, 0) {$\alpha$};
		\node at (2, 0) {$\beta$};
		\node at (4, 0) {$\gamma$};
	\end{tikzpicture}}}<
	\vcenter{\hbox{\begin{tikzpicture}[scale=0.4]
		\draw[thick] (0, 0)  (2, 0) arc(0:180:1) circle;
		\draw[fill=white, thick] (0, 0) circle [radius=15pt];
		\draw[fill=white, thick] (2, 0) circle [radius=15pt];
		\draw[fill=white, thick] (4, 0) circle [radius=15pt];
		\node at (0, 0) {$\alpha$};
		\node at (2, 0) {$\beta$};
		\node at (4, 0) {$\gamma$};
	\end{tikzpicture}}}<
	\vcenter{\hbox{\begin{tikzpicture}[scale=0.4]
		\draw[draw=none, thick] (0, 0)  (2, 0) arc(0:180:1) circle;
		\draw[fill=white, thick] (0, 0) circle [radius=15pt];
		\draw[fill=white, thick] (2, 0) circle [radius=15pt];
		\draw[fill=white, thick] (4, 0) circle [radius=15pt];
		\node at (0, 0) {$\alpha$};
		\node at (2, 0) {$\beta$};
		\node at (4, 0) {$\gamma$};
	\end{tikzpicture}}}\right)\;=\;
	\vcenter{\hbox{\begin{tikzpicture}[scale=0.4]
		\draw[thick] (0, 0)  (2, 0) arc(0:180:1) circle;
		\draw[thick] (0, 0)  (4, 0) arc(0:180:2) circle;
		\draw[fill=white, thick] (0, 0) circle [radius=15pt];
		\draw[fill=white, thick] (2, 0) circle [radius=15pt];
		\draw[fill=white, thick] (4, 0) circle [radius=15pt];
		\node at (0, 0) {$\alpha$};
		\node at (2, 0) {$\beta$};
		\node at (4, 0) {$\gamma$};
		\node at (2, -1) {$1$};
		\node at (4, -1) {$2$};
	\end{tikzpicture}}}
\]
where the numbers below indicate the order of the fusions.
\[
	\left(\vcenter{\hbox{\begin{tikzpicture}[scale=0.4]
		\draw[draw=none, thick] (0, 0)  (4, 0) arc(0:180:2) circle;
		\draw[thick] (0, 0)  (2, 0) arc(0:180:1) circle;
		\draw[thick] (0, 0)  (4, 0) arc(0:180:1) circle;
		\draw[fill=white, thick] (0, 0) circle [radius=15pt];
		\draw[fill=white, thick] (2, 0) circle [radius=15pt];
		\draw[fill=white, thick] (4, 0) circle [radius=15pt];
		\node at (0, 0) {$\alpha$};
		\node at (2, 0) {$\beta$};
		\node at (4, 0) {$\gamma$};
	\end{tikzpicture}}}<
	\vcenter{\hbox{\begin{tikzpicture}[scale=0.4]
		\draw[thick] (0, 0)  (4, 0) arc(0:180:2) circle;
		\draw[fill=white, thick] (0, 0) circle [radius=15pt];
		\draw[fill=white, thick] (2, 0) circle [radius=15pt];
		\draw[fill=white, thick] (4, 0) circle [radius=15pt];
		\node at (0, 0) {$\alpha$};
		\node at (2, 0) {$\gamma$};
		\node at (4, 0) {$\beta$};
	\end{tikzpicture}}}<
	\vcenter{\hbox{\begin{tikzpicture}[scale=0.4]
		\draw[draw=none, thick] (0, 0)  (4, 0) arc(0:180:2) circle;
		\draw[fill=white, thick] (0, 0) circle [radius=15pt];
		\draw[fill=white, thick] (2, 0) circle [radius=15pt];
		\draw[fill=white, thick] (4, 0) circle [radius=15pt];
		\node at (0, 0) {$\alpha$};
		\node at (2, 0) {$\gamma$};
		\node at (4, 0) {$\beta$};
	\end{tikzpicture}}}\right)\;=\;
	\vcenter{\hbox{\begin{tikzpicture}[scale=0.4]
		\draw[thick] (0, 0)  (2, 0) arc(0:180:1) circle;
		\draw[thick] (0, 0)  (4, 0) arc(0:180:2) circle;
		\draw[fill=white, thick] (0, 0) circle [radius=15pt];
		\draw[fill=white, thick] (2, 0) circle [radius=15pt];
		\draw[fill=white, thick] (4, 0) circle [radius=15pt];
		\node at (0, 0) {$\alpha$};
		\node at (2, 0) {$\gamma$};
		\node at (4, 0) {$\beta$};
		\node at (2, -1) {$2$};
		\node at (4, -1) {$1$};
	\end{tikzpicture}}}
\]
With this notation, we have:
\[
	\vcenter{\hbox{\begin{tikzpicture}[scale=0.4]
		\draw[draw=none, thick] (0, 0) -- (0, 2);
		\draw[thick] (0, 0)  (2, 0) arc(0:180:1) circle;
		\draw[fill=white, thick] (0, 0) circle [radius=15pt];
		\draw[fill=white, thick] (2, 0) circle [radius=15pt];
		\node at (0, 0) {$*$};
		\node at (2, 0) {$\gamma$};
		\node at (2, -1) {$1$};
	\end{tikzpicture}}}\circ_*
	\vcenter{\hbox{\begin{tikzpicture}[scale=0.4]
		\draw[draw=none, thick] (0, 0) -- (0, 2);
		\draw[thick] (0, 0)  (2, 0) arc(0:180:1) circle;
		\draw[fill=white, thick] (0, 0) circle [radius=15pt];
		\draw[fill=white, thick] (2, 0) circle [radius=15pt];
		\node at (0, 0) {$\alpha$};
		\node at (2, 0) {$\beta$};
		\node at (2, -1) {$1$};
	\end{tikzpicture}}}\;=\;
	\vcenter{\hbox{\begin{tikzpicture}[scale=0.4]
		\draw[thick] (0, 0)  (2, 0) arc(0:180:1) circle;
		\draw[thick] (0, 0)  (4, 0) arc(0:180:2) circle;
		\draw[fill=white, thick] (0, 0) circle [radius=15pt];
		\draw[fill=white, thick] (2, 0) circle [radius=15pt];
		\draw[fill=white, thick] (4, 0) circle [radius=15pt];
		\node at (0, 0) {$\alpha$};
		\node at (2, 0) {$\beta$};
		\node at (4, 0) {$\gamma$};
		\node at (2, -1) {$1$};
		\node at (4, -1) {$2$};
	\end{tikzpicture}}}\;+\;
	\vcenter{\hbox{\begin{tikzpicture}[scale=0.4]
		\draw[thick] (0, 0)  (2, 0) arc(0:180:1) circle;
		\draw[thick] (0, 0)  (4, 0) arc(0:180:2) circle;
		\draw[fill=white, thick] (0, 0) circle [radius=15pt];
		\draw[fill=white, thick] (2, 0) circle [radius=15pt];
		\draw[fill=white, thick] (4, 0) circle [radius=15pt];
		\node at (0, 0) {$\alpha$};
		\node at (2, 0) {$\gamma$};
		\node at (4, 0) {$\beta$};
		\node at (2, -1) {$2$};
		\node at (4, -1) {$1$};
	\end{tikzpicture}}}
\]
It is clear from this notation that such a picture represents a chain of noncrossing 2-partitions if each node has exactly one fusion with an element to its left (except the leftmost element) and the fusions do not cross each other.
We should notice that this looks a lot like the picture in Example~\ref{expl:rpt2tnc} except that we have an ordering on the edges.
It should also motivate the convention used in Proposition~\ref{prp:ks}

\begin{Definition}
	Let $\mathrm{P}$ a poset, and $a\leq b\in \mathrm{P}$. 
	A \emph{complete path} from $a$ to $b$ is a chain 
	$$(a=x_0<\dots<x_k=b)$$ 
	such that $x_{i-1}\leq x\leq x_i$ implies that either $x=x_{i-1}$ or $x=x_i$, let $A_{a\to b}$ the vector space spanned by the complete paths from $a$ to $b$.
	Let $p_{a\to b}=(x_0<\dots<x_k)$, and $p_{b\to c}=(y_0<\dots<y_\ell)$ complete paths from $a$ to $b$, and from $b$ to $c$.
	Then their concatenation 
	$$p_{a\to c}=p_{a\to b}.p_{b\to c}=(x_0<\dots<x_k<y_1<\dots<y_\ell)$$
	is a complete path from $a$ to $c$.\\
	Let $A_\mathrm{P}=\bigoplus A_{a\to b}$ the \emph{algebra of complete paths} where the multiplication is the concatenation of paths.
	The algebra $A_\mathrm{P}$ is graded by the length on the paths, let $A_\mathrm{P}^k$ the $k$ component with respect to this grading.
	It is clear that $A_\mathrm{P}$ is generated by the paths of length $1$, let us denote $e_{a\to b}$ the complete path $a<b$.
\end{Definition}

Let us fix a finite set $N$ of size $n$, we denote by $A$ the algebra of complete paths of $\Pi_2(N)$.

\begin{Definition}
	Let $p=(x_0<x_1)$ a complete path in $A$.
	Let us define the \emph{weight} $w(p)\in \left(\{0,\dots,n-1\}^2\right)^*$ of $p$, let $\rho_0$ and $\rho_1$ the underlying non-crossing partitions of $x_0$ and $x_1$.
	Since $p$ is a complete path, $\rho_0$ is obtained from $\rho_1$ by fusing two parts, let us denote $q_1$ and $q_2$ those two parts with $\min(q_1)\leq \min(q_2)$.
	The weight of $p$ is $(n-\min(q_1),n-\min(q_2))$.
	We extend the weight to $A$ multiplicatively, $w(p_1.p_2)=w(p_1).w(p_2)$.
	To compare the weights, we endow $\left(\{0,\dots,n-1\}^2\right)^*$ with the lexicographic order.
\end{Definition}

\begin{Example}
	Let us consider the following chain: 
	\[\vcenter
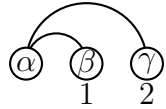
{\hbox{\begin{tikzpicture}[scale=0.4]
		\draw[thick] (0, 0)  (2, 0) arc(0:180:1) circle;
		\draw[thick] (0, 0)  (4, 0) arc(0:180:2) circle;
		\draw[fill=white, thick] (0, 0) circle [radius=15pt];
		\draw[fill=white, thick] (2, 0) circle [radius=15pt];
		\draw[fill=white, thick] (4, 0) circle [radius=15pt];
		\node at (0, 0) {$\alpha$};
		\node at (2, 0) {$\beta$};
		\node at (4, 0) {$\gamma$};
		\node at (2, -1) {$1$};
		\node at (4, -1) {$2$};
	\end{tikzpicture}}}\]
	Its weight is $(3-1,3-2).(3-1,3-3)=(2,1,2,0)$ since we first fuse the first and second elements, and then the first and the third elements.
	For:
	\[\vcenter
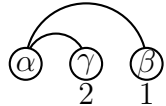
{\hbox{\begin{tikzpicture}[scale=0.4]
		\draw[thick] (0, 0)  (2, 0) arc(0:180:1) circle;
		\draw[thick] (0, 0)  (4, 0) arc(0:180:2) circle;
		\draw[fill=white, thick] (0, 0) circle [radius=15pt];
		\draw[fill=white, thick] (2, 0) circle [radius=15pt];
		\draw[fill=white, thick] (4, 0) circle [radius=15pt];
		\node at (0, 0) {$\alpha$};
		\node at (2, 0) {$\gamma$};
		\node at (4, 0) {$\beta$};
		\node at (2, -1) {$2$};
		\node at (4, -1) {$1$};
	\end{tikzpicture}}}\]
	the weight we get is $(2,0,2,1)$.
\end{Example}

\begin{Proposition}
	Let $p$ a complete path from $a$ to $b$. The weight $w(p)$ entirely determines $p$.
\end{Proposition}

\begin{proof}
	The weights contain the information of which blocks are fused at each step, thus we can recursively reconstruct $p$ from $b$.
\end{proof}

\begin{Definition}
	Let $a<c\in \Pi_2(N)$ such that the complete paths from $a$ to $c$ have length $2$.
	Let $b\in ]a,c[$ such that $w((a<b<c))$ is maximal among the complete paths from $a$ to $c$.
	Let $\mathcal{R}$ the quadratic rewriting system of $A$ given by:
	$$e_{a\to b}.e_{b\to c}=-\sum_{\beta \in]a,c[\;\text{and}\; \beta\neq b} e_{a\to \beta}.e_{\beta\to c}$$
\end{Definition}

\begin{Proposition}
	We have that $\left(A/\mathcal{R}\right)^{n-1}=h^{n-1}(\Pi_2(N))$.
\end{Proposition}

\begin{proof}
	The set of relations associated with $\mathcal{R}$ is:
	$$\sum_{b \in]a,c[} e_{a\to b}.e_{b\to c}=0$$
	which exactly gives the image of $d$ on $A^{n-1}$.
\end{proof}

We could show that this quadratic rewriting system is confluent by checking the minimal overlaps.
However, we won't need this in the sequel.

\subsection{Bijection with the planar rooted trees}

\begin{Definition}
	Let $\tau\in\mathrm{PT}(N)$. 
	We associate a chain $c(\tau)$ to $\tau$ by ordering the vertices of $\tau$ via the depth-right ordering (see Definition~\ref{def:dro}), ordering the edges of $\tau$ via the left-to-right ordering (see Definition~\ref{def:l2ro}), and setting each edge to be the fusion of two blocks.
	It is clear that this map is injective.
	The chains of the form $c(\tau)$ are the \emph{tree-like chains}.
\end{Definition}

\begin{Example}
	Let us give an example:
	\[\vcenter
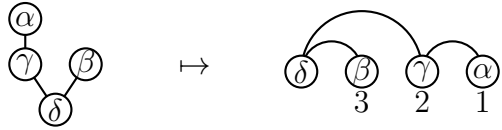
{\hbox{\begin{tikzpicture}[scale=0.4]
		\draw[thick] (0, 0.5)--(-1, 2);
		\draw[thick] (0, 0.5)--(1, 2);
		\draw[thick] (-1, 3.5)--(-1, 2);
		\draw[fill=white, thick] (0, 0.5) circle [radius=15pt];
		\draw[fill=white, thick] (-1, 2) circle [radius=15pt];
		\draw[fill=white, thick] (-1, 3.5) circle [radius=15pt];
		\draw[fill=white, thick] (1, 2) circle [radius=15pt];
		\node at (0, 0.5) {$\delta$};
		\node at (-1, 2) {$\gamma$};
		\node at (-1, 3.5) {$\alpha$};
		\node at (1, 2) {$\beta$};
	\end{tikzpicture}}}\qquad\mapsto\qquad
	\vcenter{\hbox{\begin{tikzpicture}[scale=0.4]
		\draw[thick] (0, 0)  (4, 0) arc(0:180:2) circle;
                \draw[thick] (4, 0)  (6, 0) arc(0:180:1) circle;
                \draw[thick] (0, 0)  (2, 0) arc(0:180:1) circle;
		\draw[fill=white, thick] (0, 0) circle [radius=15pt];
		\draw[fill=white, thick] (2, 0) circle [radius=15pt];
		\draw[fill=white, thick] (4, 0) circle [radius=15pt];
		\draw[fill=white, thick] (6, 0) circle [radius=15pt];
		\node at (0, 0) {$\delta$};
		\node at (2, 0) {$\beta$};
		\node at (4, 0) {$\gamma$};
		\node at (6, 0) {$\alpha$};
		\node at (2, -1) {$3$};
		\node at (4, -1) {$2$};
		\node at (6, -1) {$1$};
	\end{tikzpicture}}}
	\]
\end{Example}

\begin{Theorem}
	The normal forms of $A^{n-1}$ with respect to the rewriting system $\mathcal{R}$ are exactly the tree-like chains.
\end{Theorem}

\begin{proof}
	The rewriting system $\mathcal{R}$ is a terminating rewriting system, since it decreases along the order induced by the weight.
	Thus, the normal forms of $A^{n-1}$ with respect to the rewriting system $\mathcal{R}$ are a generating family of $\left(A/\mathcal{R}\right)^{n-1}$.
	Thus, by equality of dimension, it suffices to show that tree-like chains are normal forms which is clear since the edges are ordered right to left.
\end{proof}

In particular, we have constructed an explicit isomorphism between $h^{n-1}(\Pi_2)$ and $\mathcal{PT}$, let us denote $t$ this isomorphism.

\section{Proof of the main theorem}

\subsection{Partial Compositions in leaves}

Since we have an explicit isomorphism $t:\mathcal{PT}\to h^{n-1}(\Pi_2)$, we can define the notion of leaves on tree-like chains.

\begin{Definition}
	Let $p$ a tree-like chain.
	An element $i\in N$ is a leaf in $p$ if $i$ is a leaf in $t(p)$.
	We may notice from the definition of $t$ that $i$ is a leaf in $p$ if and only if whenever two blocks $q_1$, $q_2$ fuse in $p$, $\sigma(i)$ is never the minimal of $q_1\sqcup q_2$.
\end{Definition}

\begin{Proposition}\label{prp:leaf}
	Let $p_1$ a tree-like chain of $\Pi_2(A\sqcup\{v\})$, and $p_2$ a tree-like chain of $\Pi_2(B)$, such that $v$ is a leaf of $p_1$.
	Then $p_1\circ_v p_2$ is a tree-like chain obtained by replacing $v$ by $p_2$ in $p_1$ (up to a Koszul sign).
\end{Proposition}

\begin{proof}
	Let us compute the partial composition $p_1\circ_v p_2$.
	The underlying partition of the composition is $\pi=\{\{\alpha\}\mid \alpha\in A\}\sqcup\{B\}$.
	Let us compute the elements of $a^{-1}(\pi)$ giving a non-zero term.
	Let us the following notation for $p_1$ and $p_2$:
	\[p_1=\vcenter
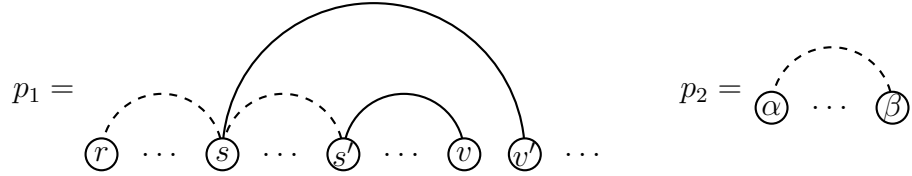
{\hbox{\begin{tikzpicture}[scale=0.4]
		\draw[thick, dashed] (0, 0)  (4, 0) arc(0:180:2) circle;
                \draw[thick] (4, 0)  (14, 0) arc(0:180:5) circle;
                \draw[thick, dashed] (4, 0)  (8, 0) arc(0:180:2) circle;
                \draw[thick] (8, 0)  (12, 0) arc(0:180:2) circle;
		\draw[fill=white, thick] (0, 0) circle [radius=15pt];
		\draw[fill=white, thick] (4, 0) circle [radius=15pt];
		\draw[fill=white, thick] (8, 0) circle [radius=15pt];
		\draw[fill=white, thick] (12, 0) circle [radius=15pt];
		\draw[fill=white, thick] (14, 0) circle [radius=15pt];
		\node at (0, 0) {$r$};
		\node at (2, 0) {$\dots$};
		\node at (4, 0) {$s$};
		\node at (6, 0) {$\dots$};
		\node at (8, 0) {$s'$};
		\node at (10, 0) {$\dots$};
		\node at (12, 0) {$v$};
		\node at (14, 0) {$v'$};
		\node at (16, 0) {$\dots$};
	\end{tikzpicture}}}\qquad 
	p_2=\vcenter{\hbox{\begin{tikzpicture}[scale=0.4]
		\draw[thick, dashed] (0, 0)  (4, 0) arc(0:180:2) circle;
		\draw[fill=white, thick] (0, 0) circle [radius=15pt];
		\draw[fill=white, thick] (4, 0) circle [radius=15pt];
		\node at (0, 0) {$\alpha$};
		\node at (2, 0) {$\dots$};
		\node at (4, 0) {$\beta$};
	\end{tikzpicture}}} 
	\]
	where a dotted line means that we have a path of several lines from one end to the other.
	We did not write down the order of the fusions since it is right to left.
	Since $\alpha$ is the minimum, it will be at the position of $v$ in the composition.
	Moreover, both the vertices of $p_1$ and of $p_2$ will retain their relative order in the composition.
	Assume that $\beta$ is larger than $v'$ in one of the terms of the composition.
	We get:
	\[\vcenter
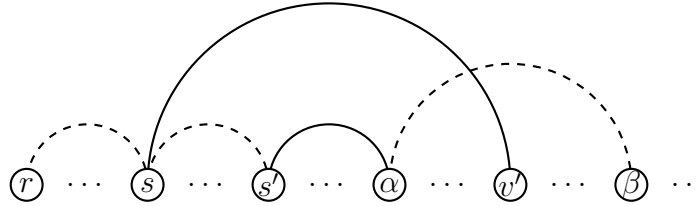
{\hbox{\begin{tikzpicture}[scale=0.4]
		\draw[thick, dashed] (0, 0)  (4, 0) arc(0:180:2) circle;
                \draw[thick] (4, 0)  (16, 0) arc(0:180:6) circle;
                \draw[thick, dashed] (4, 0)  (8, 0) arc(0:180:2) circle;
                \draw[thick] (8, 0)  (12, 0) arc(0:180:2) circle;
                \draw[thick, dashed] (12, 0)  (20, 0) arc(0:180:4) circle;
		\draw[fill=white, thick] (0, 0) circle [radius=15pt];
		\draw[fill=white, thick] (4, 0) circle [radius=15pt];
		\draw[fill=white, thick] (8, 0) circle [radius=15pt];
		\draw[fill=white, thick] (12, 0) circle [radius=15pt];
		\draw[fill=white, thick] (16, 0) circle [radius=15pt];
		\draw[fill=white, thick] (20, 0) circle [radius=15pt];
		\node at (0, 0) {$r$};
		\node at (2, 0) {$\dots$};
		\node at (4, 0) {$s$};
		\node at (6, 0) {$\dots$};
		\node at (8, 0) {$s'$};
		\node at (10, 0) {$\dots$};
		\node at (12, 0) {$\alpha$};
		\node at (14, 0) {$\dots$};
		\node at (16, 0) {$v'$};
		\node at (18, 0) {$\dots$};
		\node at (20, 0) {$\beta$};
		\node at (22, 0) {$\dots$};
	\end{tikzpicture}}}\] 
	Thus, we get a crossing, which is not possible.
	Thus, the only non-zero term is:
	\[\vcenter
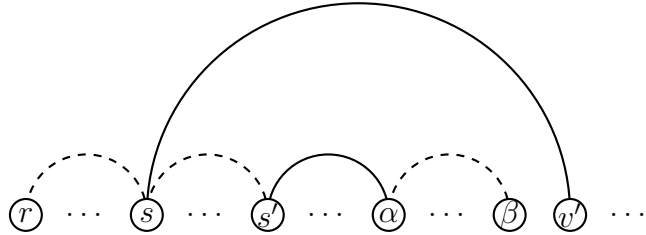
{\hbox{\begin{tikzpicture}[scale=0.4]
		\draw[thick, dashed] (0, 0)  (4, 0) arc(0:180:2) circle;
                \draw[thick] (4, 0)  (18, 0) arc(0:180:7) circle;
                \draw[thick, dashed] (4, 0)  (8, 0) arc(0:180:2) circle;
                \draw[thick] (8, 0)  (12, 0) arc(0:180:2) circle;
                \draw[thick, dashed] (12, 0)  (16, 0) arc(0:180:2) circle;
		\draw[fill=white, thick] (0, 0) circle [radius=15pt];
		\draw[fill=white, thick] (4, 0) circle [radius=15pt];
		\draw[fill=white, thick] (8, 0) circle [radius=15pt];
		\draw[fill=white, thick] (12, 0) circle [radius=15pt];
		\draw[fill=white, thick] (16, 0) circle [radius=15pt];
		\draw[fill=white, thick] (18, 0) circle [radius=15pt];
		\node at (0, 0) {$r$};
		\node at (2, 0) {$\dots$};
		\node at (4, 0) {$s$};
		\node at (6, 0) {$\dots$};
		\node at (8, 0) {$s'$};
		\node at (10, 0) {$\dots$};
		\node at (12, 0) {$\alpha$};
		\node at (14, 0) {$\dots$};
		\node at (16, 0) {$\beta$};
		\node at (18, 0) {$v'$};
		\node at (20, 0) {$\dots$};
	\end{tikzpicture}}}\]
	This is not a tree-like chain since the order of the edges is not left-to-right. 
	Indeed, the order of the edge is first left-to-right for the former edges of $p_2$, and then left-to-right for the former edges of $p_1$.
	We may check that applying the rewriting rules $\mathcal{R}$ will exchange independent fusions and thus will simply create the expected Koszul sign.
	Thus, $p_1\circ_v p_2$ is a tree-like chain obtained by replacing $v$ by $p_2$ in $p_1$ up to the Koszul sign.
\end{proof}

\begin{Example}
	Let us carry this computation on a small example:
	\[
	\vcenter
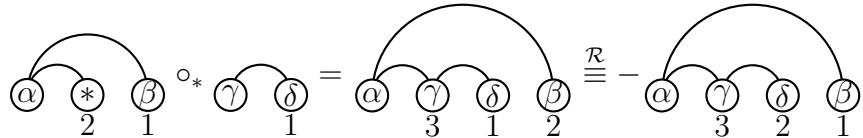
{\hbox{\begin{tikzpicture}[scale=0.4]
		\draw[draw=none, thick] (0, 0) -- (0, 3);
		\draw[thick] (0, 0)  (2, 0) arc(0:180:1) circle;
		\draw[thick] (0, 0)  (4, 0) arc(0:180:2) circle;
		\draw[fill=white, thick] (0, 0) circle [radius=15pt];
		\draw[fill=white, thick] (2, 0) circle [radius=15pt];
		\draw[fill=white, thick] (4, 0) circle [radius=15pt];
		\node at (0, 0) {$\alpha$};
		\node at (2, 0) {$*$};
		\node at (4, 0) {$\beta$};
		\node at (2, -1) {$2$};
		\node at (4, -1) {$1$};
	\end{tikzpicture}}}\circ_*
	\vcenter{\hbox{\begin{tikzpicture}[scale=0.4]
		\draw[draw=none, thick] (0, 0) -- (0, 3);
		\draw[thick] (0, 0)  (2, 0) arc(0:180:1) circle;
		\draw[fill=white, thick] (0, 0) circle [radius=15pt];
		\draw[fill=white, thick] (2, 0) circle [radius=15pt];
		\node at (0, 0) {$\gamma$};
		\node at (2, 0) {$\delta$};
		\node at (2, -1) {$1$};
	\end{tikzpicture}}}=
	\vcenter{\hbox{\begin{tikzpicture}[scale=0.4]
		\draw[thick] (0, 0)  (2, 0) arc(0:180:1) circle;
		\draw[thick] (2, 0)  (4, 0) arc(0:180:1) circle;
		\draw[thick] (0, 0)  (6, 0) arc(0:180:3) circle;
		\draw[fill=white, thick] (0, 0) circle [radius=15pt];
		\draw[fill=white, thick] (2, 0) circle [radius=15pt];
		\draw[fill=white, thick] (4, 0) circle [radius=15pt];
		\draw[fill=white, thick] (6, 0) circle [radius=15pt];
		\node at (0, 0) {$\alpha$};
		\node at (2, 0) {$\gamma$};
		\node at (4, 0) {$\delta$};
		\node at (6, 0) {$\beta$};
		\node at (2, -1) {$3$};
		\node at (4, -1) {$1$};
		\node at (6, -1) {$2$};
	\end{tikzpicture}}}\overset{\mathcal{R}}{\equiv} - 
	\vcenter{\hbox{\begin{tikzpicture}[scale=0.4]
		\draw[thick] (0, 0)  (2, 0) arc(0:180:1) circle;
		\draw[thick] (2, 0)  (4, 0) arc(0:180:1) circle;
		\draw[thick] (0, 0)  (6, 0) arc(0:180:3) circle;
		\draw[fill=white, thick] (0, 0) circle [radius=15pt];
		\draw[fill=white, thick] (2, 0) circle [radius=15pt];
		\draw[fill=white, thick] (4, 0) circle [radius=15pt];
		\draw[fill=white, thick] (6, 0) circle [radius=15pt];
		\node at (0, 0) {$\alpha$};
		\node at (2, 0) {$\gamma$};
		\node at (4, 0) {$\delta$};
		\node at (6, 0) {$\beta$};
		\node at (2, -1) {$3$};
		\node at (4, -1) {$2$};
		\node at (6, -1) {$1$};
	\end{tikzpicture}}} 
	\]
	We may notice that $-1$ is exactly the Koszul sign we are expecting when carrying this computation with trees.
\end{Example}

\subsection{Computation of the relations on the corollas}

Let us start with an intermediate result before doing the full computation.

\begin{Lemma}\label{lem:inter}
	We have:
	\begin{multline*}
	\vcenter
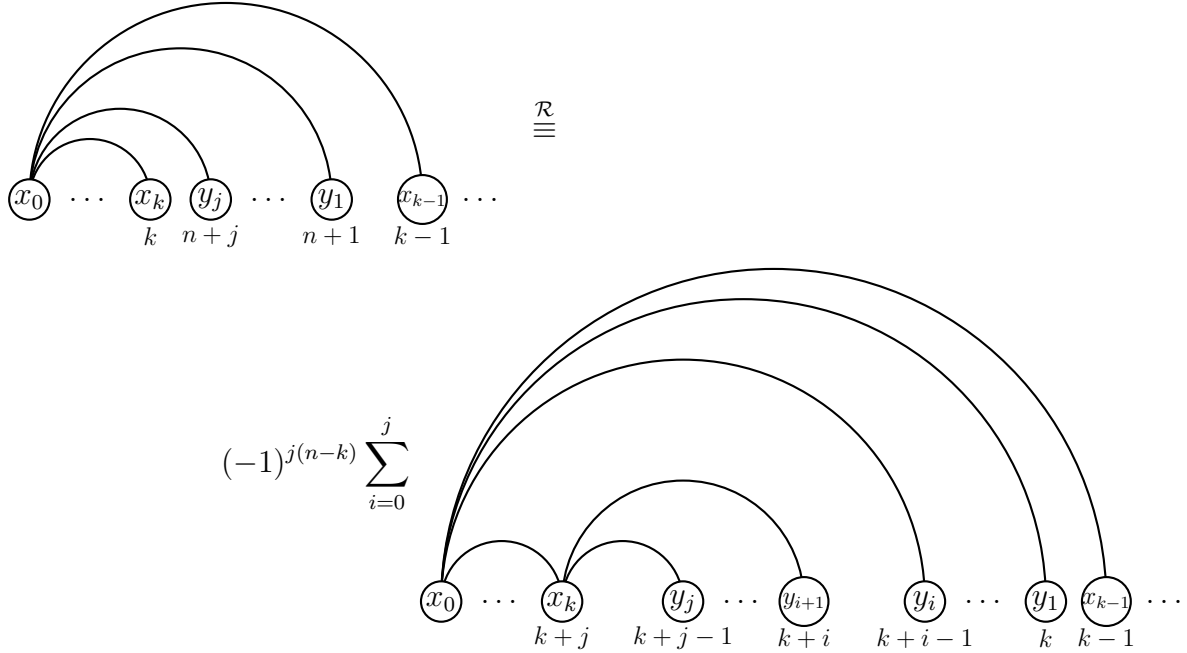
{\hbox{\begin{tikzpicture}[scale=0.4]
		\draw[thick] (0, 0)  (4, 0) arc(0:180:2) circle;
		\draw[thick] (0, 0)  (6, 0) arc(0:180:3) circle;
		\draw[thick] (0, 0)  (10, 0) arc(0:180:5) circle;
		\draw[thick] (0, 0)  (13, 0) arc(0:180:6.5) circle;
		\draw[fill=white, thick] (0, 0) circle [radius=19pt];
		\draw[fill=white, thick] (4, 0) circle [radius=19pt];
		\draw[fill=white, thick] (6, 0) circle [radius=19pt];
		\draw[fill=white, thick] (10, 0) circle [radius=19pt];
		\draw[fill=white, thick] (13, 0) circle [radius=23pt];
		\node at (0, 0) {$x_0$};
		\node at (2, 0) {$\dots$};
		\node at (4, 0) {$x_k$};
		\node at (6, 0) {$y_j$};
		\node at (8, 0) {$\dots$};
		\node at (10, 0) {$y_1$};
		\node at (13, 0) {\scalebox{0.8}{$x_{k-1}$}};
		\node at (15, 0) {$\dots$};
		\node at (4, -1.2) {\scalebox{0.8}{$k$}};
		\node at (6, -1.2) {\scalebox{0.8}{$n+j$}};
		\node at (10, -1.2) {\scalebox{0.8}{$n+1$}};
		\node at (13, -1.2) {\scalebox{0.8}{$k-1$}};
	\end{tikzpicture}}} \;\overset{\mathcal{R}}{\equiv}\\
		(-1)^{j(n-k)} \sum_{i=0}^j
	\vcenter{\hbox{\begin{tikzpicture}[scale=0.4]
		\draw[thick] (0, 0)  (4, 0) arc(0:180:2) circle;
		\draw[thick] (4, 0)  (8, 0) arc(0:180:2) circle;
		\draw[thick] (4, 0)  (12, 0) arc(0:180:4) circle;
		\draw[thick] (0, 0)  (16, 0) arc(0:180:8) circle;
		\draw[thick] (0, 0)  (20, 0) arc(0:180:10) circle;
		\draw[thick] (0, 0)  (22, 0) arc(0:180:11) circle;
		\draw[fill=white, thick] (0, 0) circle [radius=19pt];
		\draw[fill=white, thick] (4, 0) circle [radius=19pt];
		\draw[fill=white, thick] (8, 0) circle [radius=19pt];
		\draw[fill=white, thick] (12, 0) circle [radius=23pt];
		\draw[fill=white, thick] (16, 0) circle [radius=19pt];
		\draw[fill=white, thick] (20, 0) circle [radius=19pt];
		\draw[fill=white, thick] (22, 0) circle [radius=23pt];
		\node at (0, 0) {$x_0$};
		\node at (2, 0) {$\dots$};
		\node at (4, 0) {$x_k$};
		\node at (8, 0) {$y_j$};
		\node at (10, 0) {$\dots$};
		\node at (12, 0) {\scalebox{0.8}{$y_{i+1}$}};
		\node at (16, 0) {$y_i$};
		\node at (18, 0) {$\dots$};
		\node at (20, 0) {$y_1$};
		\node at (22, 0) {\scalebox{0.8}{$x_{k-1}$}};
		\node at (24, 0) {$\dots$};
		\node at (4, -1.2) {\scalebox{0.8}{$k+j$}};
		\node at (8, -1.2) {\scalebox{0.8}{$k+j-1$}};
		\node at (12, -1.2) {\scalebox{0.8}{$k+i$}};
		\node at (16, -1.2) {\scalebox{0.8}{$k+i-1$}};
		\node at (20, -1.2) {\scalebox{0.8}{$k$}};
		\node at (22, -1.2) {\scalebox{0.8}{$k-1$}};
	\end{tikzpicture}}}
	\end{multline*}
	Where on the left-hand side all the $y$'s are connected to $x_0$, we have $n+1$ vertices labelled with $x$'s from $x_0$ to $x_n$, and the order is given by first ordering the $x$'s right to left, and then the $y$'s right to left.
	On the right-hand side, the ordering is right to left (thus we get tree-like chains), and $y_1$ up to $y_i$ are connected to $x_0$, while $y_{i+1}$ up to $y_j$ are connected to $x_k$.
\end{Lemma}

\begin{proof}
	To do the computation, let us iteratively apply the rewriting rule of $\mathcal{R}$ starting with $y_1$ up to $y_i$.
	We first apply $\mathcal{R}$ between the fusions of $x_n$, and the one of $y_1$.
	Because of the fusion of $x_k$ with $x_0$, fusing the block of $x_n$ with $y_1$ would create a crossing.
	Thus, the fusion of $x_n$, and the fusion of $y_1$ are simply exchanged (with a Koszul sign appearing).
	The same is true up to $x_k$, where two terms appear, a first one where we fuse $x_k$ with $y_1$, and a second one where the fusions are simply exchanged.
	\begin{itemize}
		\item In the first case, the fusion between $y_1$ and $x_k$ forces each of the next rewriting rules to fuse the $y$'s with the block of $x_k$. 
			Thus, we get the term $i=j$.
		\item The second case reduces to the full case by relabelling the $x$'s such that $y_1$ is now $x_k$, and relabelling the $y$'s to get $j-1$ of them.
			A simple induction concludes the second case.
	\end{itemize}
	To compute the Koszul sign, we just need to count the number of rewriting rules that we have applied, we may check that this is indeed $j(n-k)$.
\end{proof}

\begin{Definition}
	Let $p$ a tree-like chain, this is a corolla if $t(p)$ is.
	We may notice from the definition of $t$ that $p$ is a corolla if and only if $p$ is constructed by fusing the largest singleton to the block containing the minimal element.
	The minimal element of $p$ is the root.
\end{Definition}

\begin{Example}
	We have: 
	\[\vcenter
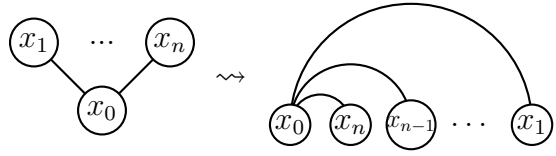
{\hbox{\begin{tikzpicture}[scale=0.6]
		\draw[thick] (0, 0.5)--(-1.5, 2);
		\draw[thick] (0, 0.5)--(1.5, 2);
		\draw[fill=white, thick] (0, 0.5) circle [radius=15pt];
		\draw[fill=white, thick] (-1.5, 2) circle [radius=15pt];
		\draw[fill=white, thick] (1.5, 2) circle [radius=15pt];
		\node at (0, 0.5) {$x_0$};
		\node at (-1.5, 2) {$x_1$};
		\node at (1.5, 2) {$x_n$};
		\node at (0, 2) {$...$};
	\end{tikzpicture}}}\;\rightsquigarrow\;
	\vcenter{\hbox{\begin{tikzpicture}[scale=0.4]
		\draw[thick] (0, 0)  (2, 0) arc(0:180:1) circle;
		\draw[thick] (0, 0)  (4, 0) arc(0:180:2) circle;
		\draw[thick] (0, 0)  (8, 0) arc(0:180:4) circle;
		\draw[fill=white, thick] (0, 0) circle [radius=19pt];
		\draw[fill=white, thick] (2, 0) circle [radius=19pt];
		\draw[fill=white, thick] (4, 0) circle [radius=23pt];
		\draw[fill=white, thick] (8, 0) circle [radius=19pt];
		\node at (0, 0) {$x_0$};
		\node at (2, 0) {$x_n$};
		\node at (4, 0) {\scalebox{0.8}{$x_{n-1}$}};
		\node at (6, 0) {$\dots$};
		\node at (8, 0) {$x_1$};
	\end{tikzpicture}}} 
	\]
\end{Example}

To conclude, it only remains to compute the composition of two corollas in their root.

\begin{Proposition}\label{prp:cor}
	Let $p_1$, $p_2$ two corollas with $r$ the root of $p_1$.
	The composition $p_1\circ_r p_2$ gives the same formula as Proposition~\ref{prp:ks}.
\end{Proposition}

\begin{proof}
	Let:
	\[p_1=
	\vcenter{\hbox{\begin{tikzpicture}[scale=0.4]
		\draw[thick] (0, 0)  (2, 0) arc(0:180:1) circle;
		\draw[thick] (0, 0)  (4, 0) arc(0:180:2) circle;
		\draw[thick] (0, 0)  (8, 0) arc(0:180:4) circle;
		\draw[fill=white, thick] (0, 0) circle [radius=19pt];
		\draw[fill=white, thick] (2, 0) circle [radius=19pt];
		\draw[fill=white, thick] (4, 0) circle [radius=23pt];
		\draw[fill=white, thick] (8, 0) circle [radius=19pt];
		\node at (0, 0) {$r$};
		\node at (2, 0) {$y_m$};
		\node at (4, 0) {\scalebox{0.7}{$y_{m-1}$}};
		\node at (6, 0) {$\dots$};
		\node at (8, 0) {$y_1$};
	\end{tikzpicture}}} \qquad,\qquad p_2=
	\vcenter{\hbox{\begin{tikzpicture}[scale=0.4]
		\draw[thick] (0, 0)  (2, 0) arc(0:180:1) circle;
		\draw[thick] (0, 0)  (4, 0) arc(0:180:2) circle;
		\draw[thick] (0, 0)  (8, 0) arc(0:180:4) circle;
		\draw[fill=white, thick] (0, 0) circle [radius=19pt];
		\draw[fill=white, thick] (2, 0) circle [radius=19pt];
		\draw[fill=white, thick] (4, 0) circle [radius=25pt];
		\draw[fill=white, thick] (8, 0) circle [radius=19pt];
		\node at (0, 0) {$x_0$};
		\node at (2, 0) {$x_n$};
		\node at (4, 0) {\scalebox{0.8}{$x_{n-1}$}};
		\node at (6, 0) {$\dots$};
		\node at (8, 0) {$x_1$};
	\end{tikzpicture}}} 
	\]
	Let $\lambda\in a^{-1}(\pi)$ such that the corresponding term is non-zero. We know that in $\lambda$, the vertices of $p_1$ are in the same relative order as in $p_1$, the vertices of $p_2$ are in the same relative order as in $p_2$, and $x_0$ is the minimum.
	The term corresponding to such a $\lambda$ is:
	\[\vcenter
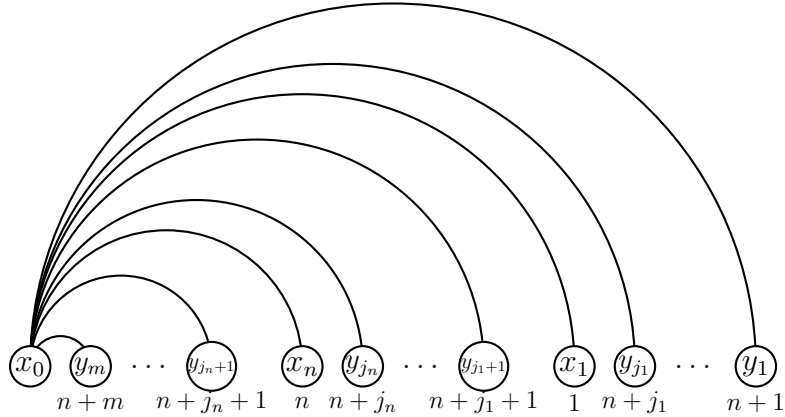
{\hbox{\begin{tikzpicture}[scale=0.4]
		\draw[thick] (0, 0)  (2, 0) arc(0:180:1) circle;
		\draw[thick] (0, 0)  (6, 0) arc(0:180:3) circle;
		\draw[thick] (0, 0)  (9, 0) arc(0:180:4.5) circle;
		\draw[thick] (0, 0)  (11, 0) arc(0:180:5.5) circle;
		\draw[thick] (0, 0)  (15, 0) arc(0:180:7.5) circle;
		\draw[thick] (0, 0)  (18, 0) arc(0:180:9) circle;
		\draw[thick] (0, 0)  (20, 0) arc(0:180:10) circle;
		\draw[thick] (0, 0)  (24, 0) arc(0:180:12) circle;
		\draw[fill=white, thick] (0, 0) circle [radius=19pt];
		\draw[fill=white, thick] (2, 0) circle [radius=19pt];
		\draw[fill=white, thick] (6, 0) circle [radius=23pt];
		\draw[fill=white, thick] (9, 0) circle [radius=19pt];
		\draw[fill=white, thick] (11, 0) circle [radius=19pt];
		\draw[fill=white, thick] (15, 0) circle [radius=23pt];
		\draw[fill=white, thick] (18, 0) circle [radius=19pt];
		\draw[fill=white, thick] (20, 0) circle [radius=19pt];
		\draw[fill=white, thick] (24, 0) circle [radius=19pt];
		\node at (0, 0) {$x_0$};
		\node at (2, 0) {\scalebox{0.9}{$y_m$}};
		\node at (4, 0) {$\dots$};
		\node at (6, 0) {\scalebox{0.7}{$y_{j_n+1}$}};
		\node at (9, 0) {$x_n$};
		\node at (11, 0) {\scalebox{0.9}{$y_{j_n}$}};
		\node at (13, 0) {$\dots$};
		\node at (15, 0) {\scalebox{0.7}{$y_{j_1+1}$}};
		\node at (18, 0) {$x_1$};
		\node at (20, 0) {\scalebox{0.9}{$y_{j_1}$}};
		\node at (22, 0) {$\dots$};
		\node at (24, 0) {$y_1$};
		\node at (2, -1.2) {\scalebox{0.8}{$n+m$}};
		\node at (6, -1.2) {\scalebox{0.8}{$n+j_n+1$}};
		\node at (9, -1.2) {\scalebox{0.8}{$n$}};
		\node at (11, -1.2) {\scalebox{0.8}{$n+j_n$}};
		\node at (15, -1.2) {\scalebox{0.8}{$n+j_1+1$}};
		\node at (18, -1.2) {\scalebox{0.8}{$1$}};
		\node at (20, -1.2) {\scalebox{0.8}{$n+j_1$}};
		\node at (24, -1.2) {\scalebox{0.8}{$n+1$}};
	\end{tikzpicture}}}
	\]
	Iteratively applying Lemma~\ref{lem:inter} to each group $y_{j_\ell +1}$ to $y_{j_{\ell+1}}$ from right to left concludes the proof.
\end{proof}

\begin{Theorem}
	The operadic structure on $\Lambda h^\bullet(\Pi_2)$ is isomorphic to the Brace operad.
\end{Theorem}

\begin{proof}
	It suffices to apply Theorem~\ref{thm:carac} using the Propositions~\ref{prp:leaf} and~\ref{prp:cor}.
\end{proof}

An interesting consequence of this theorem is that $\Lambda\check{h}^\bullet(\Pi_2)$ is naturally endowed with brace products by the construction of \cite{DD25}.
Indeed, in \cite{DD25}, $\Lambda\check{h}^\bullet(\Pi_2)$ is endowed with a structure of left $\Lambda h^\bullet(\Pi_2)$-module.
We recall that a structure of left module over an operad $\mathcal{P}$ is exactly the same as a structure of $\mathcal{P}$-algebra, thus we have a $\mathrm{Brace}$-algebra structure on $\Lambda\check{h}^\bullet(\Pi_2)$.

\section*{Acknowledgement}

{\small The author is funded by a postdoctoral fellowship of the ERC Starting Grant “Low Regularity Dynamics via Decorated Trees” (LoRDeT) of Yvain Bruned (grant agreement No. 101075208).\\
The author would like to thank Christophe Hohlweg, Loïc Foissy, Claudia Malvenuto, and Ruggero Bandiera for organising the conference Algebraic Combinatorics and Finite Groups IV where most of this work has been done.
The author would also like to thank Bérénice Delcroix-Oger for the related discussions at the conference Algebraic Combinatorics and Finite Groups IV. 
}

\bibliographystyle{plain}
\bibliography{Bibly}

\end{document}